\documentclass[reqno,12pt]{amsart}
\usepackage{amscd,amsfonts,mathrsfs,amsthm,enumerate}
\usepackage{amssymb, amsmath}
\usepackage{stmaryrd}
\usepackage{epsfig}
\usepackage[alphabetic, sorted-cites]{amsrefs}

\thanks{The first author is co-funded by the European Social Found (ESF) and National Sources, in
the framework of the program \textquotedblleft Support of Post-Doctoral Fellows\textquotedblright\, of the
\textquotedblleft Operational Program Education and Life Long Learning\textquotedblright\, of the Hellenic Ministry
of Education, Life Long Learning and Religious Affairs.}


\def\mric{\operatorname{Ric}_M}

\def\real     #1{{\mathbb R^{#1}}}

\def\uu       #1#2#3{{#1}^{#2#3}}

\def\equationcolor {\color{black}}
\def\textcolor     {\color{black}}

\def\bcoleq    {\begin{equation}\equationcolor}
\def\ecoleq    {\textcolor\end{equation}}
\def\bcoleqn   {\equationcolor\begin{eqnarray}}
\def\ecoleqn   {\end{eqnarray}\textcolor}

\def\gm{{\operatorname{g}_M}}
\def\gn{{\operatorname{g}_N}}

\def\gk{{\operatorname{g}_{M\times N}}}
\def\rm{{\operatorname{R}_M}}
\def\rn{{\operatorname{R}_N}}
\def\rk{{\operatorname{R}_{M\times N}}}
\def\sk{{\operatorname{s}_{M\times N}}}
\def\rind{\operatorname{R}}
\def\sind{\operatorname{s}}
\def\Sind{\operatorname{S}}

\def\T{\operatorname{T}}
\def\Aut{\operatorname{Aut}}
\def\P{\operatorname{P}}
\def\Q{\operatorname{Q}}
\def\gE{{\operatorname{g}_{E}}}
\def\dF{\operatorname{d}\hspace{-3pt}F}
\def\df{\operatorname{d}\hspace{-3pt}f}

\def\gind{\operatorname{g}}

\def\Gind{\operatorname{G}}
\def\smin{\sigma}
\def\eig{\operatorname{Eig}}
\def\sym{\operatorname{Sym}}

\DeclareMathOperator*{\Ric}{Ric}

\DeclareMathOperator*{\Id}{Id}

\DeclareMathOperator*{\rank}{rank}

\newtheorem{theorem}{Theorem}[section]
\newtheorem{mythm}{Theorem}

\newtheorem{lemma}[theorem]{Lemma}

\newtheorem{definition}[theorem]{Definition}
\theoremstyle{definition}
\newtheorem{remark}[theorem]{Remark}
\newtheorem{example}[theorem]{Example}

\newcommand{\bfig}{\begin{figure}}
\newcommand{\tsum}{\textstyle\sum}
\newcommand{\efig}{\end{figure}}

\makeatletter
\def\pproof#1{\@ifnextchar[\opargproof
{\opargproof[\it Proof of #1.]}}
\def\opargproof[#1]{\par\noindent {\bf #1 }}

\makeatother

\numberwithin{equation}{section}

\begin{document}
\title[Strong Maximum Principles]{The Strong Elliptic Maximum Principle For Vector Bundles \\
and Applications to Minimal Maps
}
\author[Andreas Savas-Halilaj]{\textsc{Andreas Savas-Halilaj}}
\author[Knut Smoczyk]{\textsc{Knut Smoczyk}}
\address{Andreas Savas-Halilaj\newline
Leibniz Universit\"at Hannover\newline
Institut f\"ur Differentialgeometrie\newline
Welfengarten 1\newline
30167 Hannover\newline
Germany\newline
{\sl E-mail address:} {\bf savasha@math.uni-hannover.de}
}
\address{Knut Smoczyk\newline
Leibniz Universit\"at Hannover\newline
Institut f\"ur Differentialgeometrie\newline
Welfengarten 1\newline
30167 Hannover\newline
Germany\newline
{\sl E-mail address:} {\bf smoczyk@math.uni-hannover.de}
}
\date{}

\begin{abstract}
Based on works by Hopf, Weinberger, Hamilton and Evans, we state and prove the strong
elliptic maximum principle for smooth sections in vector bundles over Riemannian manifolds
and give some applications in Differential Geometry.
Moreover, we use this maximum principle to obtain various rigidity theorems and Bernstein
type theorems in higher codimension for minimal maps between Riemannian manifolds.
\end{abstract}

\maketitle

\section{Introduction}
The maximum principle is one of the most powerful tools used in the theory
of PDEs  and Geometric Analysis. In general, maximum principles
for solutions of second order elliptic differential equations, that are defined in the closure of a
bounded domain of the euclidean space, appear in two forms.
The \textit{weak maximum principle} states that
the maximum of the solution is attained at the boundary of the domain, but in principle it
might occur in the interior as well. On the other hand, the \textit{strong maximum principle} asserts that the
solution achieves its maximum only at boundary points, unless it is constant. For instance,
H. Hopf \cite{hopf} established such strong maximum principles for a wide class of general
second order differential equations. For example, he proved that if a solution $u$ of
the uniformly elliptic differential equation
\begin{gather}\label{elliptic}
\mathscr{L}u=0,\quad\mathscr{L}=\sum_{i,j=1}^{m}a^{ij}\partial^2_{ij}+\sum_{j=1}^{m}b^j\partial_j, \tag{$\ast$}
\end{gather}

attains its supremum or infimum at an interior point
of its domain $D$ of definition, then it must be constant.

Equivalently, the above strong elliptic maximum principle of Hopf can be interpreted as follows: If
a solution $u$ of $\mathscr{L}u=0$ maps
an interior point of $D$ to the boundary of the set $K=(\inf_{D}u,\sup_{D}u)$, then $u$ maps
any point of $D$ to the boundary of $K$ and hence it must be constant. For the proof of
this strong maximum principle Hopf used the \textit{Hopf Lemma}, which implies that the subset
$B\subset D$ consisting of points where $u$ attains a value in $\partial K$ is open. Since by continuity
$B$ is also closed, one has $B=D$, if $D$ is connected and $B$ is non-empty.

The generalization of Hopf's maximum principle to elliptic and semi-linear parabolic systems has
been first considered by H. Weinberger \cite{weinberger}. Let us recall briefly here the elliptic
version of this strong maximum principle: Assume that the vector valued map $$u:D\subset\real{m}\to\real{n},\quad u:=(u_1,\dots,u_n),$$
is a solution of the differential system
$$\mathscr{L}u+\Psi(u)=0,$$
such that $u(D)$ is contained in a closed convex set $K\subset\real{n}$.
Here $\mathscr{L}$ is a second order uniformly elliptic
differential operator of the form given in (\ref{elliptic}), $D$ is an open domain of $\real{m}$ and $\Psi:\real{n}\to\real{n}$ is a Lipschitz
continuous map. Suppose further that for any boundary point $y_0\in\partial K$ the vector $\Psi(y_0)$
belongs to the tangent cone of $K$ at $y_0$ (for the exact definition see Section $2.1$). Under various additional
assumptions on the regularity of the boundary of the convex set $K$, Weinberger proved that, if an interior point of $D$ is
mapped via $u$ to a boundary point of $K$, then every point of $D$ is mapped to the boundary of $K$. Recently,
L.C. Evans \cite{evans1} gave a proof of Weinberger's maximum principle without imposing any regularity assumption on the boundary of the convex set $K$.

In his seminal papers, R. Hamilton \cite{hamilton2,hamilton1} derived parabolic maximum principles for sections in Riemannian vector bundles. There one compares the solution of a parabolic differential equation with a solution of an associated ODE. The weak parabolic maximum principle of Weinberger
can be seen as a special case of Hamilton's more general maximum principle in \cite{hamilton1}
since Weinberger's result follows from the application of Hamilton's maximum principle in the
case of a trivial bundle.
Hamilton's maximum principle appears in many different forms and became an important tool
in the study of geometric evolution equations (cf. \cite{ecker,ni1,brendle,andrews}).

Here we state and prove the strong elliptic maximum principle for sections in Riemannian vector bundles. This maximum principle is in the most general form and contains all the previous results by Hopf, Weinberger, Evans and it also
contains the elliptic version of Hamilton's parabolic
maximum principle. It turns out that it is extremely powerful and
we apply it to derive optimal Bernstein type results for minimal maps between Riemannian manifolds.

In order to state the elliptic version of the strong maximum principle for sections in vector bundles,
we must introduce an appropriate notion of convexity for subsets of Riemannian vector bundles. In
\cite{hamilton2} Hamilton gave the following definition:

\begin{definition}{\bf(Hamilton).}
Let $(E,\pi,M)$ be a vector bundle over the manifold $M$ and let $K$ be a closed subset of $E$.
\begin{enumerate}[(i)]

\item
The set $K$ is said to be fiber-convex or convex in the fiber, if for each point $x$ of $M$, the
set $K_x:=K\cap E_x$ is a convex subset of the fiber $E_x=\pi^{-1}(x)$.
\medskip
\item
The set $K$ is said to be invariant under parallel transport, if for every smooth
curve $\gamma:[0,b]\to M$ and any vector $v\in K_{\gamma(0)}$, the unique parallel section
$v(t)\in E_{\gamma(t)}$, $t\in [0,b]$, along $\gamma(t)$ with $v(0)=v$, is contained
in $K$.
\medskip
\item
A fiberwise map $\Psi:E\to E$ is a map such  that $\pi\circ\Psi=\pi$, where $\pi$ denotes the bundle projection. We say a fiberwise map $\Psi$ points
into $K$ (or is inward pointing), if for any $x\in M$ and any
$\vartheta\in \partial K_x$ the vector $\Psi(\vartheta)$ belongs to
the tangent cone $C_{\vartheta}K_x$ of $K_x$ at $\vartheta$.
\end{enumerate}
\end{definition}

Next we state the strong elliptic  maximum principle for sections in Riemannian vector bundles.
Throughout the paper all manifolds will be smooth and connected.

Let $(E,\pi,M)$ be a vector bundle of rank $k$ over a smooth manifold $M$. Suppose $\gind_E$ is a
bundle metric on $E$ and that $\nabla$ is a metric connection on $E$.
In this paper
we consider uniformly elliptic operators $\mathscr{L}$ on $\Gamma(E)$ of second order
that are given locally by
\begin{equation}\label{dast}
\mathscr{L}=\sum_{i,j=1}^ma^{ij}\nabla^{2}_{e_{i},e{_{j}}}+\sum_{j=1}^mb^{j}\nabla_{e_{j}}, \tag{$\ast\ast$}
\end{equation}
where $a\in\Gamma(TM\otimes TM)$ is a symmetric, uniformly positive definite tensor and
$b\in\Gamma(TM)$ is a smooth vector field such that
$$a=\sum_{i,j=1}^m\uu aij e_i\otimes e_j\quad\text{and}\quad b=\sum_{j=1}^mb^je_j$$
in a local frame field $\{e_1,\dots,e_k\}$ of $TM$.

\begin{mythm}{\bf (Strong Elliptic Maximum Principle).}\label{mp1}\\
Let $(M,\gm)$ be a Riemannian manifold
and $(E,\pi,M)$ a vector bundle
over $M$ equipped with a Riemannian metric $\gind_{E}$ and a metric connection $\nabla$.
Let $K$ be a closed fiber-convex subset of the bundle $E$ that is invariant under parallel
transport and let $\phi\in K$ be  a smooth section such that
\begin{equation*}
\mathscr{L}\phi+\Psi(\phi)=0,
\end{equation*}
where here $\mathscr{L}$ is a uniformly elliptic operator of second order of the form given in (\ref{dast})  and $\Psi$ is a smooth
fiberwise map that points into $K$. If there exists a point $x_0$ in the interior of $M$ such that $\phi(x_0)\in\partial K_{x_0}$, then $\phi(x)\in\partial K_x$ for any point $x\in M$. If, additionally, $K_{x_0}$ is strictly convex at $\phi(x_0)$, then $\phi$ is a parallel section.
\end{mythm}

For the classification of minimal maps between Riemannian manifolds we will later use
a special case of Theorem \ref{mp1} for
smooth, symmetric tensors $\phi\in\sym(E^*\otimes E^*)$. Before stating the result let us
recall the following definition due to Hamilton \cite[Section 9]{hamilton2}.

\begin{definition}{\bf (Hamilton).}\label{def null}
A fiberwise map $\Psi:\sym(E^*\otimes E^*)\to\sym(E^*\otimes E^*)$ is said to satisfy the null-eigenvector
condition, if whenever $\vartheta$ is a non-negative symmetric $2$-tensor at a point $x\in M$ and
if $v\in T_{x}M$ is a null-eigenvector of $\vartheta$, then $\Psi(\vartheta)(v,v)\ge 0$.
\end{definition}

The next theorem is the elliptic analogue of the
maximum principle of Hamilton \cite[Lemma 8.2, p. 174]{hamilton1}.
More precisely:

\begin{mythm}\label{mp2}
Let $(M,\gm)$ be a Riemannian manifold and $(E,\pi,M)$ a Riemannian vector bundle
over $M$ equipped with a metric connection.
Suppose that $\phi\in\sym(E^*\otimes E^*)$ is non-negative definite
and satisfies
$$\mathscr{L}\phi+\Psi(\phi)=0,$$
where here $\Psi$ is a smooth fiberwise map satisfying the null-eigenvector condition.
If there is an interior point of $M$ where $\phi$ has a zero-eigenvalue, then
$\phi$ must have a zero-eigenvalue everywhere. Additionally, if $\phi$ vanishes identically  at an
interior point of $M$, then $\phi$ vanishes everywhere.
\end{mythm}

Since the maximum principle for scalar functions has uncountable many applications in Geometric
Analysis we expect that the strong maximum principle for sections in vector bundles will have plenty of applications as well. In Section \ref{sec3} we will apply this strong maximum principle to derive a classification of minimal maps between Riemannian manifolds.

Before stating our results in this direction, let us introduce some new definitions. Let $(M,\gm)$ and
$(N,\gn)$ be two Riemannian manifolds of dimensions $m$ and $n$ respectively.
For any smooth map $f:M\to N$
its differential $\df$ induces a map $\Lambda^{k}\df:\Lambda^{k}T^*M\to\Lambda^{k}T^*N$ given by
\begin{equation*}
\left(\Lambda^{k}\df\right)(v_1\,,\cdots,v_k):=\df(v_1)\wedge\cdots\wedge\df(v_k),
\end{equation*}
for any smooth vector fields $v_1,\dots,v_k\in TM$.
The map $\Lambda^{k}\df$ is called the $k$-\textit{Jacobian} of $f$. The \textit{supremum norm} or
the $k$-\textit{dilation} $\|\Lambda^{k}\df\|(x)$ of the map $f$ at a point $x\in M$
is defined as the supremum of
$$\sqrt{\det\big([f^{\ast}\gn(v_i,v_j)]_{1\le i,j\le k}\bigr)}$$
when $\{v_1,\dots,v_m\}$ runs over all orthonormal bases of $T_xM$.
The $k$-dilation measures how much the map stretches $k$-dimensional volumes. The map
$f:M\to N$ is called \textit{weakly $k$-volume decreasing} if $\|\Lambda^k\df\|\le 1$, \textit{strictly
$k$-volume decreasing} if $\|\Lambda^k\df\|<1$ and \textit{$k$-volume preserving} if $\|\Lambda^k\df\|=1$.
As usual for $k=1$ we use the term \textit{length} instead of $1$-volume and if $k=2$ we use the term
\textit{area} instead of $2$-volume. The map $f:M\to N$ is called an \textit{isometric immersion}, if $f^*\gn=\gm$.
A smooth map $f:M\to N$ is called \textit{minimal}, if its graph
$$\Gamma(f):=\{(x,f(x))\in M\times N:x\in M\}$$
is a minimal submanifold of $(M\times N,\gk=\gm\times\gn)$.

One of the main objectives in the present article is to prove the following results:
\begin{mythm}\label{thmD}
Let $(M,\gm)$ and $(N,\gn)$ be two Riemannian manifolds. Suppose $M$ is compact, $m=\dim M\ge 2$
and that there exists a constant $\smin>0$
such that the sectional curvatures $\sigma_M$ of $M$ and $\sigma_N$ of $N$ and
the Ricci curvature $\mric$ of $M$ satisfy
$$\sigma_M\,\,>\,\, -\smin,\quad\,\,\frac{1}{m-1}\mric\,\,\ge\,\,\smin\,\,\ge\,\,\sigma_N.$$
If $f:M\to N$ is a minimal map that is weakly length decreasing, then one of the following
holds:
\begin{enumerate}[(i)]
\item
$f$ is constant.
\item
$f$ is an isometric immersion, $M$ is Einstein with  $\mric=(m-1)\sigma$ and the
restriction of $\sigma_N$ to $\df(TM)$ is equal to $\sigma$.
\end{enumerate}
In particular,
any strictly length decreasing minimal map is constant.
\end{mythm}

A similar statement holds in the case of weakly area decreasing maps.
\begin{mythm}\label{thmC}
Let $(M,\gm)$ and $(N,\gn)$ be two Riemannian manifolds. Suppose $M$ is compact, $m=\dim M\ge 2$
and that there exists a constant $\smin>0$
such that the sectional curvatures $\sigma_M$ of $M$ and $\sigma_N$ of $N$ and
the Ricci curvature $\mric$ of $M$ satisfy
$$\sigma_M\,\,>\,\, -\smin,\quad\,\,\frac{1}{m-1}\mric\,\,\ge\,\,\smin\,\,\ge\,\,\sigma_N.$$
If $f:M\to N$ is a smooth minimal map that is weakly area decreasing, then one of the following
holds:
\begin{enumerate}[(i)]
\item
$f$ is constant.
\item
There exists
a non-empty closed set $D$ such that $f$ is an isometric immersion on $D$ and $f$ is strictly
area decreasing on the complement
of $D$. Moreover, $M$ is Einstein on $D$ with  $\mric=(m-1)\sigma$ and the
restriction of $\sigma_N$ to $\df(TD)$ is equal to $\sigma$.
\end{enumerate}
In particular, any strictly area decreasing minimal map is constant
and any area preserving minimal map is an isometric immersion.
\end{mythm}

In the special case where the manifold $N$ is one-dimensional we have the following stronger theorem:

\begin{mythm}\label{thmE}
Let $(M,\gm)$ and $(N,\gn)$ be two Riemannian manifolds. Suppose that $M$ is compact, $\dim M\ge 2$,
$\Ric_{M}>0$ and that $\dim N=1$. Then any smooth minimal map $f:M\to N$ is constant.
\end{mythm}

As pointed out in the final remarks of Section $3.6$, Theorems \ref{thmD}, \ref{thmC} and \ref{thmE} are optimal
in various ways. We include some examples and remarks concerning the imposed
assumptions at the end of the paper.

The paper is organized as follows. In Section \ref{sec2}
we recall the strong maximum principle for uniformly elliptic
systems of second order by Weinberger-Evans and give the proofs of
Theorems \ref{mp1} and \ref{mp2}. The geometry of
graphs will be treated in Section \ref{sec3}, where we also derive the crucial formula needed
in the proof of Theorems \ref{thmD}, \ref{thmC} and \ref{thmE}.


\section{Strong elliptic maximum principles for sections in vector bundles}\label{sec2}
In this section we shall derive strong elliptic maximum principles for smooth sections in
Riemannian vector bundles. The original idea goes back to
the fundamental work of Hamilton \cite{hamilton2,hamilton1} on the Ricci flow, where a strong
parabolic maximum principle for symmetric tensors and weak
parabolic maximum principles for sections in vector bundles were proven.

\subsection{Convex sets}
In this subsection we review the basic definitions about the geometry of convex sets
in euclidean space such as supporting half-spaces, tangent cones and normal vectors.
A brief exposition can be found in the book by Andrews and Hopper \cite[Appendix B]{andrews}.

Recall that a subset $K$ of $\real{n}$ is called \textit{convex} if for any pair of points $z$, $w\in K$,
the segment
$$\mathcal{E}_{z,w}:=\{tz+(1-t)w\in\real{n}:t\in(0,1)\}$$
is contained in $K$. The set $K$ is said to be \textit{strictly convex}, if for
any pair $z,w\in K$ the segment $\mathcal{E}_{z,w}$ belongs to the interior of $K$.

A convex set $K\subset\real{n}$ may have non-smooth boundary. Hence, there is no well-defined tangent
or normal space of $K$ in the classical sense. However, there is a way to generalize these important
notions for closed convex subsets of $\real{n}$. This difficulty can be overcome by using the property
that points lying outside of the given set can be separated from the set itself by half-spaces. This
property, leads to the notion of generalized tangency.

Let $K$ be a closed convex subset of the euclidean space $\real{n}$. A supporting half-space of the set $K$ is
a closed half-space of $\real{n}$ which contains $K$ and has points of $K$ on its boundary. A supporting
hyperplane of $K$ is a hyperplane which is the boundary of a supporting half-space of $K$.
The \textit{tangent cone} $C_{y_0}K$ of $K$
at $y_0\in \partial K$ is defined as the intersection of all supporting half-spaces of $K$ that contain $y_0$.

We may also introduce the notion of normal vectors
to the boundary of a closed convex set. Let $K\subset\real{n}$ be a closed convex subset and $y_0\in\partial K$. Then
\begin{enumerate}[(i)]
\item
A non-zero vector $\xi$ is called \textit{normal vector} of $\partial K$ at $y_0$, if $\xi$ is normal to a
supporting hyperplane of $K$ passing through $y_0$. This normal vector is called \textit{inward pointing}, if it points into the half-space containing the set $K$.
\medskip
\item
A vector $\eta$ is called \textit{inward pointing} at $y_0\in\partial K$, if
$$\langle\xi,\eta\rangle\ge 0$$
for any inward pointing normal vector $\xi$ at $y_0$. Here, $\langle\cdot,\cdot\rangle$ denotes the usual inner product in $\real{n}$.
\end{enumerate}

\subsection{Maximum principles for systems}
In \cite{weinberger}, H. Weinberger  established a strong maximum principle for vector valued
maps with values in a convex set $K\subset\real{n}$ whose boundary $\partial K$ satisfies regularity
conditions that he called \textquotedblleft{\textit{slab conditions}}\textquotedblright. Inspired by
the ideas of Weinberger, X. Wang in \cite{wang4} gave a geometric proof of the strong maximum principle,
in the case where the boundary of $K$ is of class $C^2$. The idea of Wang was to apply
the classical maximum principle of Hopf to the function $d(u):D\to\real{}$, whose value at $x$ is
equal to the distance of $u(x)$ from the boundary $\partial K$ of $K$. Very recently, L.C. Evans \cite{evans1}
was able to remove all additional regularity requirements on the boundary of the convex set $K$ by showing that even
if $d(u)$ is not twice differentiable, it is still a viscosity super-solution of an appropriate partial
differential equation. The argument of Evans is completed by applying a strong maximum principle
due to F. Da Lio \cite{dalio} for viscosity super-solutions of partial differential equations.

\begin{theorem}{\bf (Weinberger-Evans).}\label{weinberger}
Let $K$ be a closed, convex set of $\real{n}$ and $u:D\subset\real{m}\to K\subset\real{n}$ a
solution of the uniformly elliptic system of partial differential equations
$$(\mathscr{L}u)(x)+\Psi(x,u(x))=0,\quad x\in D,$$
where $D$ is a domain of $\real{m}$, $\Psi:D\times\real{n}\to\real{n}$ is a continuous map that is locally Lipschitz continuous in the second variable and $\mathscr{L}$ is a uniformly elliptic operator given in (\ref{elliptic}). Suppose that
\begin{enumerate}[(i)]
\item
there is a point $x_0$ in the interior of $D$ such that
$u(x_0)\in\partial K$,
\medskip
\item
for any $(x,y)\in D\times\partial K$, the vector $\Psi(x,y)$ points into $K$ at the point $y\in\partial K$.
\end{enumerate}
Then $u(x)\in\partial K$ for any $x\in D$. If $\partial K$ is strictly convex at $u(x_0)$, then $u$ is constant.
\end{theorem}

\begin{remark}
The above maximum principle is not valid without the convexity assumption. We illustrate this
by an example. Let
$$D=\{(x,y)\in\real{2}:x^2+y^2<1\}$$
be the unit open disc in $\real{2}$ and let $h:\partial D\to\Gamma\subset\real{2}$ be a
continuous function that maps $\partial D$ onto the upper semicircle
$$\Gamma=\{(x,y)\in\real{2}:x^2+y^2=1\,\text{and}\,y\ge 0\}.$$
Denote now by $u:D\to\real{2}$ the solution of the Dirichlet problem with boundary data
given by the function $h$. Let us examine the image of the harmonic map $u$.
We claim at first that the image of $u$ is contained in the convex hull
$\mathcal{C}(\Gamma)$ of the upper semicircle. That is,
$$K:=u\left(\overline{D}\right)\subset\,\mathcal{C}(\Gamma)=\{(x,y)\in\real{2}:x^2+y^2\le 1\,\text{and }y\ge 0\}.$$
Arguing indirectly, let us assume that this is not true. The convex hull $\mathcal{C}(K)$
of $K$ contains $\mathcal{C}(\Gamma)$. Since $K$ is compact, the set
$\mathcal{C}(K)$ is also compact. Consequently, there exist a point $(x_0,y_0)$ in $D$ such that
$u(x_0,y_0)\in\partial\,\mathcal{C}(K)$ and $u(x_0,y_0)\not\in\partial\,\mathcal{C}(\Gamma)$.
Then, from the maximum principle of Weinberger-Evans we deduce that $u(x,y)\in\partial\,\mathcal{C}(K)$ for any
$(x,y)\in D$. This contradicts with the boundary data imposed by the Dirichlet condition.
Therefore, $K$ is contained in $\mathcal{C}(\Gamma)$. From Theorem \ref{weinberger},
we conclude that there is no common point of $K$ with the $x$-axes. Hence, $K$ is not convex.
The same argument yields that there is no point of $D$ which is mapped to $\Gamma$ via $u$. On
the other hand, because $K$ is compact, there are infinitely many points of $D$ which are mapped
to the boundary of $K$. Furthermore, we claim that the set $K$ has non-empty interior.
In order to show this, suppose to the contrary that $K\setminus\partial K=\emptyset$. Then,
$$\rank(\operatorname{d}\hspace{-2pt}u)\le 1$$
which implies that the closure of the set $u(D)$ is a continuous curve $L$ joining the points
$(-1,0)$ and $(1,0)$. But then, the continuity of $u$ leads to a contradiction. Indeed, for any
sequence $\{p_{k}\}_{k\in\mathbb{N}}$ of points of $D$ converging to a point $p\in u^{-1}(0,1)$, we have
$\lim u(p_k)\neq (0,1)$.
\end{remark}

\subsection{Maximum principles for sections in vector bundles}
Here we give
the analogue version of the Weinberger-Evans strong maximum principle for sections
in Riemannian vector bundles.
Our approach is inspired by ideas developed by Weinberger \cite{weinberger} and Hamilton \cite{hamilton2,hamilton1}.

For the proof of the strong maximum principle we will use a beautiful result
due to C. B\"{o}hm and B. Wilking \cite[Lemma 1.2, p. 670]{bohm}.

\begin{lemma}{\bf (B\"{o}hm-Wilking).}\label{wilking}
Suppose that $M$ is a Riemannian manifold and that $(E,\pi,M)$ is a Riemannian vector bundle over
$M$ equipped with a metric connection. Let $K$ be a closed and fiber-convex subset
of the bundle $E$ that is invariant under parallel transport. If $\phi$ is a smooth section with
values in $K$ then, for any $x\in M$ and $v\in T_xM$, the Hessian
$$\nabla^{2}_{v,v}\phi=\nabla_{v}\nabla_v\phi-\nabla_{\nabla_{v}v}\phi$$
belongs to the tangent cone of $K_x$ at the point $\phi(x)$.
\end{lemma}

The following result is an immediate consequence of the above lemma.

\begin{lemma}\label{wilking2}
Suppose that $M$ is a Riemannian manifold and that $(E,\pi,M)$ is a Riemannian vector bundle over
$M$ equipped with a metric connection. Let $K$ be a closed and fiber-convex subset
of $E$ that is invariant under parallel transport. If $\phi$ is a smooth section with
values in $K$ then, for any $x\in M$, the vector $(\mathscr{L}\phi)(x)$ belongs to the tangent cone
$C_{\phi(x)}K_x$, for any operator $\mathscr{L}$ of the form given in (\ref{dast}).
\end{lemma}

Now we derive the proof of the strong elliptic maximum principle formulated in Theorem \ref{mp1}.

{\bf Proof of Theorem \ref{mp1}.}
Let $\{\phi_1,\dots,\phi_k\}$ be a geodesic
orthonormal frame field of smooth sections in $E$, defined in a sufficiently small neighborhood
$U$ of a local trivialization around $x_0\in M$. Hence,
$$\phi=\sum_{i=1}^{k}u_{i}\phi_{i}$$
where $u_{i}:U\to\real{}$, $i\in\{1,\dots,k\}$, are smooth functions.

With respect to this frame we have
\begin{eqnarray*}
\mathscr{L}\phi&=&\sum^k_{i=1}\Big\{\mathscr{L} u_{i}+\big(\text{gradient terms of }u_i\big)
+\sum\limits^{k}_{j=1}u_{j}\,\gE(\mathscr{L}\phi_j,\phi_i)\Big\}\phi_i \\
&=&-\sum^{k}_{i=1}\gE(\Psi(\phi),\phi_i)\phi_i
\end{eqnarray*}

Therefore, the map $u:U\to\real{k}$, $u=(u_1,\dots,u_k)$, satisfies a uniformly elliptic system of second order of the form
\begin{equation}\label{pde}
\mathscr{\tilde L}u+\Phi(u)=0,
\end{equation}
where here $\Phi:\real{k}\to\real{k}$,
$$\Phi:=(\Phi_1,\dots,\Phi_k),$$
is given by
\begin{equation}\label{newphi}
\Phi_{i}(u)=\gE\left(\Psi\left(\tsum\limits_{j=1}^k u_j\phi_j\right)+\tsum\limits_{j=1}^{k}u_j\mathscr{L}\phi_j\,,\phi_i\right),
\end{equation}
for any $i\in\{1,\dots,k\}$.

Consider the convex set
$$\mathcal{K}:=\{(y_1,\dots,y_k)\in\real{k}:\tsum\limits_{i=1}^{k}y_{i}\phi_{i}(x_0)\in K_{x_0}\}.$$
\textit{{\bf Claim 1:} For any point $x\in U$ we have $u(x)\in\mathcal{K}$}.

Indeed, fix a point $x\in U$ and
let $\gamma:[0,1]\to U$ be the geodesic curve joining the points $x$ and $x_0$. Denote by $\theta$ the
parallel section which is obtained by the parallel transport of $\phi(x)$ along the geodesic
$\gamma$. Then,
$$\theta\circ\gamma=\sum_{i=1}^{k}y_{i}\,\phi_{i}\circ\gamma,$$
where $y_{i}:[0,1]\to\real{}$, $i\in\{1,\dots,k\}$, are smooth functions. Because,
$\theta$ and $\phi_{i}$, $i\in\{1,\dots,k\}$ are parallel along $\gamma$, it follows
that
\begin{eqnarray*}
0=\nabla_{\partial_{t}}(\theta\circ\gamma)=\sum_{i=1}^{k}y'_{i}(t)\phi_{i}(\gamma(t)).
\end{eqnarray*}
Hence, $y_{i}(t)=y_{i}(0)=u_{i}(x),$ for any $t\in [0,1]$ and $i\in\{1,\dots,k\}$. Therefore,
$$\theta(\gamma(1))=\theta(x_0)=\sum_{i=1}^{k}u_{i}(x)\phi_{i}(x_0).$$
Since by our assumptions $K$ is invariant under parallel transport, it follows that
$\theta(x_0)\in K_{x_0}$. Hence, $u(U)\subset\mathcal{K}$ and this proves Claim 1.

\textit{{\bf Claim 2:} For any  $y\in\partial \mathcal{K}$ the vector $\Phi(y)$ as
defined in $(\ref{newphi})$ points
into $\mathcal{K}$ at $y$.}

First note that the boundary of each slice $K_x$ is invariant under parallel transport.
From (\ref{newphi}) we deduce that it suffices to prove that  both terms appearing on
the right hand side of (\ref{newphi}) point into $\mathcal{K}$. The first term points into
$\mathcal{K}$ by assumption on $\Psi$. The second term is inward pointing due to
Lemma \ref{wilking2} by B\"ohm and Wilking.
This completes the proof of  Claim 2.

The solution of the uniformly second order elliptic partial differential system $(\ref{pde})$ satisfies
all the assumptions of Theorem \ref{weinberger}. Therefore, because $u(x_0)\in\partial\mathcal{K}$
it follows that $u(U)$ is contained in the boundary $\partial\mathcal{K}$ of $\mathcal{K}$. Consequently,
$\phi(x)\in\partial K$ for any $x\in U$. Since $M$ is connected, we deduce that $\phi(M)\subset\partial K$.

Note, that if $\mathcal{K}$ is additionally strictly convex at
$u(x_0)$, then the map $u$ is constant. This
implies that
$$\phi(x)=\sum_{i=1}^{k}u_i(x_0)\phi_i(x)$$
for any $x\in U$. Consequently, $\phi$ is a parallel section taking all its values in $\partial K$.
This completes the proof of Theorem \ref{mp1}. \hfill{$\square$}

\begin{remark}
Theorem \ref{mp1} implies the following: Suppose the fibers of $K$ are cones with vertices at $0$
and that they are strictly convex at $0$.  If $\phi(x)=0$ in a point $x\in M$, then $\phi$ vanishes everywhere.
\end{remark}

We can now prove Theorem \ref{mp2}.

{\bf Proof of Theorem \ref{mp2}.}
Let $K$ be the set of all non-negative definite symmetric $2$-tensors on $M$, i.e.
$$K:=\{\vartheta\in\sym(E^*\otimes E^*):\vartheta\ge 0\}.$$
Each fiber $K_{x}$ is a closed convex cone that is strictly convex at $0$. Moreover, $K$ is invariant under parallel transport.
The set of all boundary points of $K_x$ is given by
$$\partial K_{x}=\{\vartheta\in K_{x}:\exists\text{ a non-zero }v\in T_{x}M\text{ such that }
\vartheta(v,v)=0\}.$$
It is a classical fact in Convex Analysis (see for example the book \cite[Appendix B]{andrews}), that the tangent cone of
$K_{x}$ at a point $\vartheta$ of its boundary is given by
$$C_{\vartheta}K_{x}=\{\psi\in\sym(E_{x}^*\otimes E_{x}^*):
\psi(v,v)\ge 0,\forall\, v\in E_x \text{ with }\vartheta(v,v)=0\}.$$
Thus $\psi$ is in the tangent cone of $K_x$ at $\vartheta$, if and only if it satisfies the null-eigenvector
condition of Hamilton given in Definition \ref{def null}.
By Theorem \ref{mp1} we immediately get the proof of Theorem \ref{mp2}.
\hfill$\square$

\subsection{A second derivative criterion for symmetric $2$-tensors}
For $\phi\in\sym(E^*\otimes E^*)$ a real number $\lambda$ is called \textit{eigenvalue} of $\phi$ with
respect to $\gE$ at the point $x\in M$, if there exists a non-zero vector $v\in E_{x}$, such that
\begin{equation*}
\phi(v,w)=\lambda\gE(v,w),
\end{equation*}
for any $w\in E_{x}$. The linear subspace $\eig_{\lambda,\phi}(x)$ of $E_x$ given by
\begin{equation*}
\eig_{\lambda,\phi}(x):=\{v\in E_x:\phi(v,w)=\lambda\gE(v,w),\, \text{for any}\,w\in E_x\},
\end{equation*}
is called the \textit{eigenspace} of $\lambda$ at $x$. Since $\phi$ is symmetric it admits $k$ real
eigenvalues $\lambda_1(x),\dots,\lambda_k(x)$ at each point $x\in M$. We will always arrange the
eigenvalues such that $\lambda_{1}(x)\leq\cdots\leq\lambda_{k}(x)$.

\begin{theorem}{\bf(Second Derivative Criterion)}\label{test}
Let $(M,\gm)$ be a Riemannian manifold and $(E,\pi,M)$ a Riemannian vector bundle
of rank $k$ over the manifold $M$ equipped with a metric connection $\nabla$.
Suppose that $\phi\in\sym(E^*\otimes E^*)$ is a smooth symmetric $2$-tensor. If the
biggest eigenvalue $\lambda_k$ of $\phi$ admits a local maximum $\lambda$ at an interior
point $x_0\in M$, then
$$(\nabla\phi)(v,v)=0\quad\text{and}\quad(\mathscr{L}\phi)(v,v)\le 0,$$
for all vectors $v$ in the eigenspace of $\lambda$ at $x_0$ and for all uniformly elliptic
second order operators $\mathscr{L}$.
\end{theorem}

{\bf Remark.} Replacing $\phi$ by $-\phi$ in Theorem \ref{test} gives a similar result for
the smallest eigenvalue $\lambda_1$ of $\phi$.

\begin{proof}
Let $v\in \eig_{\lambda,\phi}(x_0)$ be a unit vector and $V\in\Gamma(E)$ a smooth section such that
\begin{equation*}
V(x_0)=v\quad \text{and}\quad(\nabla V)(x_0)=0.
\end{equation*}
Define the symmetric $2$-tensor $\Sind$ given by $\Sind:=\phi-\lambda\gE$.
From our assumptions, the symmetric $2$-tensor $\Sind$ is non-positive definite in a small neighborhood
of $x_0$. Moreover, the biggest eigenvalue of $\Sind$ at $x_0$ equals $0$. Consider the smooth function
$f:M\to\real{}$, given by
\begin{equation*}
f(x):=\Sind(V(x),V(x)).
\end{equation*}
The function $f$ is non-positive in the same neighborhood around $x_0$
and attains a local maximum in an interior point $x_0$. In particular,
\begin{equation*}
f(x_0)=0,\quad \df(x_0)=0\quad\text{and}\quad(\mathscr{L}f)(x_0)\le 0.
\end{equation*}
Consider a local orthonormal frame field $\{e_1,\dots,e_m\}$ with respect to $\gm$ defined in a neighborhood
of the point $x_0$ and assume that the expression of $\mathscr{L}$ with respect to this frame is
\begin{equation*}
\mathscr{L}=\sum_{i,j=1}^ma^{ij}\nabla^{2}_{e_{i},e{_{j}}}+\sum_{j=1}^mb^{j}\nabla_{e_{j}}.
\end{equation*}
A simple calculation yields
\begin{equation*}
\nabla_{e_{i}}f=\df(e_i)=\left(\nabla_{e_{i}}\Sind\right)(V,V)+2\Sind\left(\nabla_{e_{i}}V,V\right).
\end{equation*}
Taking into account that $\gE$ is parallel, we deduce that
$$0=(\nabla f)(x_0)=(\nabla\Sind)(v,v)=(\nabla\phi)(v,v).$$
Furthermore,
\begin{eqnarray*}
\nabla^{2}_{e_{i},e_{j}}f&=&(\nabla^{2}_{e_{i},e_{j}}\Sind)(V,V)
+2\Sind(V,\nabla^{2}_{e_{i},e_{j}}V)\\
&&+2\left(\nabla_{e_i}\Sind\right)(\nabla_{e_j}V,V)+2\left(\nabla_{e_j}\Sind\right) (\nabla_{e_i}V,V)\\
&&+2\Sind(\nabla_{e_i}V,\nabla_{e_j}V).
\end{eqnarray*}
Bearing in mind the definition of $S$ and using the fact that $\gE$ is parallel with respect to $\nabla$,
we obtain
\begin{eqnarray*}
\mathscr{L}f&=&(\mathscr{L}\phi)(V,V)+2\Sind(V,\mathscr{L}V)\\
&&+\sum_{i,j=1}^{m}2a^{ij}\left\{\Sind(\nabla_{e_i}V,\nabla_{e_j}V)
+2(\nabla_{e_i}\Sind)(\nabla_{e_j}V,V)\right\}\\
&=&(\mathscr{L}\phi)(V,V)+2\Sind(V,\mathscr{L}V)\\
&&+\sum_{i,j=1}^{m}2a^{ij}\left\{\Sind(\nabla_{e_i}V,\nabla_{e_j}V)
+2(\nabla_{e_i}\Sind)(\nabla_{e_j}V,V)\right\}.
\end{eqnarray*}
Estimating at $x_0$ and taking into account that $V(x_0)=v$ is a null eigenvector
of $\Sind$ at $x_0$, we get
$$0\ge(\mathscr{L}f)(x_0)=(\mathscr{L}\phi)(v,v).$$
This completes the proof.
\end{proof}

\subsection{An application}
In order to demonstrate how to apply the strong elliptic maximum principle and the second
derivative criterion, we shall give here an example in the case of hypersurfaces in euclidean space.

Let $M$ be an oriented $m$-dimensional hypersurface of $\real{m+1}$. Denote by
$\xi$ a unit normal vector field along the hypersurface. The most natural symmetric $2$-tensor on
$M$ is the \textit{scalar second fundamental form} $h$ of the hypersurface with respect
to the unit normal direction $\xi$, that is
$$h(v,w):=-\langle \operatorname{d}\hspace{-2pt}\xi(v),w\rangle,$$
for any $v,w\in TM$.
The eigenvalues
$$\lambda_1\le\cdots\le\lambda_m$$
of $h$ with respect to the induced metric $\gind$ are called the \textit{principal curvatures} of the hypersurface.
The quantity $H$ given by
$$H:=\lambda_1+\cdots+\lambda_m$$
is called the \textit{scalar mean curvature} of the hypersurface and the function $\|h\|$
given by
$$\|h\|^2:=\lambda^2_1+\cdots+\lambda^{2}_m$$
is called the
\textit{norm of the second fundamental form} with respect to the metric $\gind$. It is well
known that if $h$ is non-negative definite, then $M$ is locally the boundary
of a convex subset of $\real{m+1}$. For this reason, the hypersurface $M$ is called \textit{convex} whenever
its scalar second fundamental form is non-negative definite.

In the sequel we will give an alternative short proof of a well-known theorem,
first proven by W. S\"uss \cite{suss}.
\begin{theorem}{\bf{(S\"uss)}}\label {thmalex}
Any closed and convex hypersurface $M$ in $\real{m+1}$ with constant mean curvature is a round
sphere.
\end{theorem}
\begin{proof}
The Laplacian of the second fundamental form $h$ with respect to the induced Riemannian metric $\gind$,
is given by Simons' identity \cite{simons}
\begin{equation}\label{simons}
\Delta h+\|h\|^2h-Hh^{(2)}=0\,,
\end{equation}
where
$$h^{(2)}(v,w):=\operatorname{trace}\bigl(h(v,\cdot\,)\otimes h(w,\cdot\,)\bigr).$$
Since the manifold $M$ is closed, there exists an interior point $x_0\in M$, where the smallest
principal curvature $\lambda_1$ attains its global minimum $\lambda_{\min}$. Recall that by convexity we
have that $\lambda_{\min}\ge0$.

The fiberwise map $\Psi$ given by
$$\Psi(\vartheta)=\|\vartheta\|^2\vartheta-H\vartheta^{(2)},$$
obviously satisfies the null-eigenvector condition.

If $\lambda_{\min}=0$, then due to Theorem B, the smallest principal curvature
of $M$ vanishes everywhere. Hence, $\rank h<m$. It is a well known fact in
Differential Geometry that the set
$$M_{0}:=\{x\in M:\rank h_{x}={\max}_{z\in M}\rank h_{z}\},$$
is open and dense in $M$ (a standard reference is \cite{ferus}).
From the Codazzi equation, it follows that the nullity distribution
$$\mathcal{D}:=\{v\in TM_0:h(v,w)=0,\,\,\text{for all }w\in TM_0\},$$
is integrable and its integrals are totally geodesic submanifolds of $M$. On
the other hand, the Gau{\ss} formula says that these submanifolds are totally geodesic in
$\real{m+1}$. Moreover, because $M$ is complete it follows that these submanifolds must be
also complete. This contradicts with the assumption of compactness. Consequently,
the minimum $\lambda_{\min}$ of the smallest principal curvature must be strictly positive.

Let $v$ be a unit eigenvector of $h$ corresponding to $\lambda_{\min}$ at $x_0$. Applying
Theorem \ref{test}, we obtain
\begin{eqnarray*}
0&\ge&\|h\|^2(x_0)\lambda_{\min}-H\lambda_{\min}^2 \\
&=&\lambda_{\min}\left(\|h\|^2(x_0)-H\lambda_{\min}\right),
\end{eqnarray*}
Because $\|h\|^2\ge H^2/m$, we deduce that
$$\|h\|^2(x_0)-H\lambda_{\min}\ge H\left(H/m-\lambda_{\min}\right)\ge 0.$$
Consequently,
$$0\ge \lambda_{\min}H\left(H/m-\lambda_{\min}\right)\ge 0,$$
and so $H/m=\lambda_{\min}$. On the other hand $\lambda_{\min}$ is the global minimum of all
principal curvatures on $M$ and $H$ is constant. This shows that
the smallest principal curvature $\lambda_1(x)$ at an arbitrary point $x\in M$ satisfies
$$\lambda_{\min}\le\lambda_1(x)\le H/m=\lambda_{\min}.$$
Therefore $M$ is everywhere umbilic. It is well-known that the only closed and totally
umbilic hypersurfaces are the round spheres.
\end{proof}

\section{Bernstein Type Theorems for Minimal Maps}\label{sec3}
In this section we shall develop the relevant geometric identities for graphs induced by
smooth maps $f:M\to N$. Moreover, we will derive estimates that
will be crucial in the proofs of Theorems \ref{thmD}, \ref{thmC} and \ref{thmE}.

According to the Bernstein theorem \cite{bernstein}, all complete minimal graphs in the three
dimensional euclidean space are generated by affine maps. This result cannot be extended to complete minimal
graphs in any euclidean space without imposing further assumptions. There is a very rich and long literature
concerning complete minimal graphs which are generated by maps between euclidean
spaces, marked by works of W. Fleming \cite{fleming}, S.S. Chern and R. Osserman \cite{chern1}, J. Simons
\cite{simons}, E. Bombieri, E. de Giorgi and E. Giusti \cite{bombieri}, R. Schoen, L. Simon and S.T. Yau
\cite{schoen1}, S. Hildebrandt, J. Jost and K.-O. Widmann \cite{hildebrandt} and many others.

In the last decade there have been obtained several Bernstein type theorems in higher codimension,
see for instance \cite{swx}, \cite{li}, \cite{h-s2, h-s1} and \cite{jxy}.

The generalized Bernstein type problem that we are investigating here is to determine under which
additional geometric conditions minimal graphs generated by maps $f:M\to N$ are totally
geodesic. There are several recent results involving mean curvature flow in the case where both $M$ and $N$
are compact. For instance, we mention \cites{wang2,wang3,wang1}, \cite{sw}, \cite{tsui1}
and \cite{lee}. In these papers the authors prove that the mean curvature flow of graphs,
generated by smooth maps $f:M\to N$ satisfying suitable conditions,
evolves $f$ to a  constant map or an isometric immersion as time tends to infinity. This implies in particular
Bernstein results for minimal graphs satisfying the same conditions as the initial map.

\subsection{Geometry of graphs}

Let $(M,\gm)$ and $(N,\gn)$ be Riemannian manifolds of dimension $m$ and $n$, respectively.
The induced metric on the product manifold will be denoted by
$$\gk=\gm\times \gn.$$
A smooth map
$f:M\to N$ defines an embedding $F:M\to M\times N$, by
$$F(x)=\bigl(x,f(x)\bigr),\quad x\in M.$$
The graph of $f$ is defined to be the submanifold $\Gamma(f):=F(M)$. Since $F$ is an embedding,
it induces another Riemannian metric
$\gind:=F^*\gk$
on $M$. The two natural projections
$$\pi_{M}:M\times N\to M\,,\quad \pi_{N}:M\times N\to N$$
are submersions,
that is they are smooth and have maximal rank. Note that the tangent bundle  of the product manifold
$M\times N$, splits as a direct sum
\begin{equation*}
T(M\times N)=TM\oplus TN.
\end{equation*}
The four metrics $\gm,\gn,\gk$ and $\gind$ are related by
\begin{eqnarray}
\gk&=&\pi_M^*\gm+\pi_N^*\gn\,,\label{met1}\\
\gind&=&F^*\gk=\gm+f^*\gn\,.\label{met2}
\end{eqnarray}
Additionally, let us define the symmetric $2$-tensors
\begin{eqnarray}
\sk&:=&\pi_M^*\gm-\pi_{N}^*\gn\,,\label{met3}\\
\sind&:=&F^*\sk=\gm-f^*\gn\,.\label{met4}
\end{eqnarray}
Note that $\sk$ is a semi-Riemannian metric of signature $(m,k)$ on the manifold $M\times N$.

The Levi-Civita connection $\nabla^{\gk}$ associated to the Riemannian metric $\gk$ on $M \times N$ is
related to the Levi-Civita connections $\nabla^{\gm}$ on $(M,\gm)$ and $\nabla^{\gn}$ on
$(N,\gn)$ by
$$\nabla^{\gk}=\pi_M^*\nabla^{\gm}\oplus\pi_N^*\nabla^{\gn}\,.$$
The corresponding curvature operator $\rk$ on $M\times N$ with respect to the metric
$\gk$ is related to the curvature
operators $\rm$ on $(M,\gm)$ and $\rn$ on $(N,\gn)$ by
\begin{equation*}
\rk=\pi^{*}_{M}\rm\oplus\pi^{*}_{N}\rn.
\end{equation*}
Denote the Levi-Civita connection on $M$ with respect to the induced metric
$\gind=F^*\gk$ simply by $\nabla$ and the curvature tensor by $\rind$.

On the manifold $M$ there are many interesting bundles. The most important
one is the \textit{pull-back bundle} $F^{\ast}T(M\times N)$.
The differential $\dF$ of the map $F$ is a section in $F^{\ast}T(M\times N)\otimes T^*M$.
The covariant derivative of it is called the \textit{second fundamental tensor}
$A$ of the graph. That is,
\begin{equation*}
A(v,w):=(\widetilde\nabla\hspace{-2pt}\dF)(v,w)=\nabla^{\gk}_{\dF(v)}\dF(w)-\dF(\nabla_vw)
\end{equation*}
where $v,w\in TM$, $\widetilde\nabla$ is the induced connection on $F^{\ast}T(M\times N)\otimes T^*M$
and $\nabla$ is the Levi-Civita connection associated to the Riemannian metric
$$\gind:=F^*\gk.$$
The trace of $A$ with respect to the metric $\gind$ is called the
{\it mean curvature vector field} of $\Gamma(f)$ and it
will be denoted by
$$\vec{H}:=\operatorname{trace}A.$$
Note that $\vec{H}$ is a section in the normal bundle of the
graph. If $\vec{H}$ vanishes identically the graph is said to be
minimal. Following Schoen's \cite{schoen} terminology, a map $f:M\to N$
between Riemannian manifolds is called \textit{minimal} if its graph
$\Gamma(f)$ is a minimal submanifold of the product space $(M\times
N,\gk)$.

By \textit{Gau\ss' equation} the curvature tensors $\rind$ and $\rk$
are related by the formula
\begin{eqnarray}
\rind(v_1,w_1,v_2,w_2)&=&(F^*\rk)(v_1,w_1,v_2,w_2)\nonumber\\
&&+\gk\bigl( A(v_1,v_2),A(w_1,w_2)\bigr)\nonumber\\
&&-\gk\bigl( A(v_1,w_2),A(w_1,v_2)\bigr),\label{gauss}
\end{eqnarray}
for any $v_1,v_2,w_1,w_2\in TM$. Moreover, the second fundamental form satisfies the
\textit{Codazzi equation}
\begin{eqnarray}
(\nabla_uA)(v,w)-(\nabla_vA)(u,w)&=&\rk\bigl(\dF(u),\dF(v)\bigr)\dF(w)\nonumber\\
&&-\dF\bigl(\rind(u,v)w\bigr),\label{codazzi}
\end{eqnarray}
for any $u,v,w$ on $TM$.

\subsection{Singular decomposition}
In this subsection we closely follow the notations used in \cite{tsui1}.
For a fixed point $x\in M$, let
$$\lambda^2_{1}(x)\le\cdots\le\lambda^2_{m}(x)$$
be the eigenvalues of $f^{*}\gn$ with respect to $\gm$. The corresponding values $\lambda_i\ge 0$,
$i\in\{1,\dots,m\}$, are usually called
\textit{singular values} of the differential $\df$ of $f$ and give rise to continuous functions on $M$. Let
$$r=r(x)=\rank\df(x).$$
Obviously, $r\le\min\{m,n\}$ and $\lambda_{1}(x)=\cdots=\lambda_{m-r}(x)=0.$
At the point $x$ consider an orthonormal basis $\{\alpha_{1},\dots,\alpha_{m-r};\alpha_{m-r+1},
\dots,\alpha_{m}\}$ with respect to $\gm$ which diagonalizes $f^*\gn$. Moreover, at
$f(x)$ consider an orthonormal basis $\{\beta_{1},\dots,\beta_{n-r};\beta_{n-r+1},\dots,\beta_{n}\}$
with respect to $\gn$ such that
$$\df(\alpha_{i})=\lambda_{i}(x)\beta_{n-m+i},$$
for any $i\in\{m-r+1,\dots,m\}$. The above procedure is called the \textit{singular decomposition}
of the differential $\df$.

Now we are going to define a special basis for the tangent and the normal space of the graph
in terms of the singular values. The vectors
\begin{equation}
e_{i}:=\left\{
\begin{array}{ll}
\alpha _{i}, & 1\le i\le m-r,\\
&  \\
\frac{1}{\sqrt{1+\lambda _{i}^{2}\left( x\right) }}\left( \alpha
_{i}\oplus \lambda _{i}\left(x\right) \beta _{n-m+i}\right) , & m-r+1\leq
i\leq m,
\end{array}
\right.\label{tangent}
\end{equation}
form an orthonormal basis with respect to the metric $\gk$ of the tangent space $\dF\left(T_{x}M\right)$ of the graph $\Gamma(f)$ at
$x$. Moreover, the vectors
\begin{equation}
\xi_{i}:=\left\{
\begin{array}{ll}
\beta _{i}, & 1\leq i\leq n-r,\\
&  \\
\frac{1}{\sqrt{1+\lambda _{i+m-n}^{2}\left( x\right) }}\left( -\lambda
_{i+m-n}(x)\alpha _{i+m-n}\oplus \beta _{i}\right) , & n-r+1\leq i\leq n, \\
\end{array}
\right.\label{normal}
\end{equation}
give an orthonormal basis with respect to  $\gk$ of the normal space $\mathcal{N}_{x}M$ of the
graph $\Gamma(f)$ at the point $f(x)$. From the formulas above, we deduce that
\begin{equation}
\sk(e_{i},e_{j})=\frac{1-\lambda^{2}_{i}(x)}{1+\lambda^{2}_{i}(x)}\delta_{ij},\quad 1\le i,j\le m.
\end{equation}
Consequently, the eigenvalues of the $2$-tensor $\sind$ with respect to $\gind$, are
$$\frac{1-\lambda^{2}_{1}(x)}{1+\lambda^{2}_{1}(x)}\ge\cdots\ge\frac{1-\lambda^{2}_{m-1}(x)}{1+\lambda^{2}_{m-1}(x)}
\ge\frac{1-\lambda^{2}_{m}(x)}{1+\lambda^{2}_{m}(x)}.$$
Moreover,
\begin{eqnarray}
\hspace{-.5cm}
\sk(\xi_{i},\xi_{j})
&=&\begin{cases}
\displaystyle
-\delta_{ij},&\,1\le i\le n-r\\[4pt]\displaystyle
-\frac{1-\lambda^{2}_{i+m-n}(x)}{1+\lambda^{2}_{i+m-n}(x)}\delta_{ij},&\, n-r+1\le i\le n.
\end{cases}\label{normal}
\end{eqnarray}
and
\begin{equation}
\sk(e_{m-r+i},\xi_{n-r+j})=-\frac{2\lambda_{m-r+i}(x)}{1+\lambda^{2}_{m-r+i}(x)}\delta_{ij},\quad 1\le i,j\le r.\label{mixed}
\end{equation}

\subsection{Area decreasing maps}
Recall that a map $f:M\to N$ is \textit{weakly area decreasing} if
$\|\Lambda^{2}\df\|\le 1$ and \textit{strictly area decreasing} if $\|\Lambda^{2}\df\|<1$.
The above notions are expressed in terms of the singular values by the inequalities
$$\lambda_{i}^2(x)\lambda_{j}^2(x)\le1\quad\text{and}\quad\lambda_{i}^2(x)\lambda_{j}^2(x)<1,$$
for any $1\le i<j\le m$ and $x\in M$, respectively. On the other hand, the sum of two eigenvalues of
the tensor $\sind$ with respect to $\gind$ equals
$$\frac{1-\lambda^{2}_{i}}{1+\lambda^{2}_{i}}+\frac{1-\lambda^{2}_{j}}{1+\lambda^{2}_{j}}=
\frac{2(1-\lambda^{2}_{i}\lambda^{2}_{j})}{(1+\lambda^{2}_{i})(1+\lambda^{2}_{j})}.$$
Hence, the strictly area-decreasing property of the map $f$ is equivalent to the $2$-\textit{positivity} of
the symmetric tensor $\sind$.

From the algebraic point of view, the $2$-positivity of a symmetric tensor $\T\in\sym(T^*M\otimes T^*M)$ can
be expressed as the convexity of another symmetric tensor $\T^{[2]}\in\sym(\Lambda^{2}T^*M\otimes\Lambda^{2}T^*M)$.
Indeed, let $\P$ and $\Q$ be two symmetric $2$-tensors. Then, the map $\P \varowedge\Q$ given
by
\begin{eqnarray*}
(P\varowedge\Q)(v_1\wedge w_1,v_2\wedge w_2)&=&\P(v_1,v_2)\Q(w_1,w_2)+\P(w_1,w_2)\Q(v_1,v_2) \\
&-&\P(w_1,v_2)\Q(v_1,w_2)-P(v_1,w_2)\Q(w_1,v_2)
\end{eqnarray*}
gives rise to an element of $\sym(\Lambda^{2}T^*M\otimes\Lambda^{2}T^*M)$. The operator $\varowedge$ is
called the \textit{Kulkarni-Nomizu product}.
Now we assign to each symmetric $2$-tensor $\T\in\sym(T^*M\otimes T^*M)$  an element $\T^{[2]}$ of the bundle
$\sym(\Lambda^{2}T^*M\otimes\Lambda^{2}T^*M)$, by setting
$$\T^{[2]}:=\T\varowedge\gind.$$
The Riemannian metric $\Gind$ of the bundle $\Lambda^{2}TM$ is related
to the Riemannian metric $\gind$ on the manifold $M$ by the  formula
$$\Gind=\tfrac{1}{2}\gind\varowedge\gind=\tfrac{1}{2}\gind^{[2]}.$$
The relation between the eigenvalues of $\T$ and the eigenvalues of
$\T^{[2]}$ is explained in the following lemma:
\begin{lemma}
Suppose that $\T$ is a symmetric $2$-tensor with eigenvalues $\mu_{1}\le\cdots\le\mu_{m}$
and corresponding eigenvectors $\{v_1,\dots,v_{m}\}$  with
respect to
$\gind$. Then the eigenvalues of the symmetric $2$-tensor $\T^{[2]}$ with respect to $\Gind$
are
$$\mu_{i}+\mu_{j},\quad 1\le i<j\le m,$$
with corresponding eigenvectors
$$v_{i}\wedge v_{j},\quad 1\le i<j\le m.$$
\end{lemma}

\subsection{A Bochner-Weitzenb\"{o}ck formula}
Our next goal is to compute the Laplacians of the tensors $\sind$ and $\sind^{[2]}$. The next computations
closely follow those for a similarly defined tensor in \cite{sw}.
In order to control the smallest eigenvalue of $\sind$, let us define the symmetric $2$-tensor $$\Phi_c:=\sind-\frac{1-c}{1+c}\gind,$$
where $c$ is a non-negative constant.

At first let us compute the covariant derivative of the tensor $\Phi_{c}$. Since $\nabla\hspace{-3pt}\gind=0$ and $\nabla^{\gk}\sk=0$, we have
\begin{eqnarray*}
(\nabla_v\Phi_c)(u,w)&=&(\nabla_v\sind)(u,w) \\
&=&(\nabla^{\gk}_{\dF(v)}\sk)\bigr(\dF(u),\dF(w)\bigl) \\
&&+\sk\bigl(A(v,u),\dF(w)\bigr)+\sk\bigl(\dF(u),A(v,w)\bigr) \\
&=&\sk\bigl(A(v,u),\dF(w)\bigr)+\sk\bigl(\dF(u),A(v,w)\bigr),
\end{eqnarray*}
for any tangent vectors $u,v,w\in TM$.

Now let us compute the Hessian of $\Phi_c$.
Differentiating once more gives
\begin{eqnarray}
&&\left(\nabla^2_{v_1,v_2}\Phi_c\right)\hspace{-4pt}(u,w)\nonumber\\
&&\quad=\sk\bigl((\nabla_{v_1}A)(v_2,u),\dF(w)\bigr)
+\sk\bigl(A(v_2,u),A(v_1,w)\bigr)\nonumber\\
&&\quad\quad+\sk\bigl(A(v_1,u),A(v_2,w)\bigr)+\sk\bigl(\dF(u),(\nabla_{v_1}A)(v_2,w)\bigr),
\nonumber
\end{eqnarray}
for any tangent vectors $v_1,v_2,u,w\in TM$. Applying Codazzi's equation (\ref{codazzi}) and exploiting the
symmetry of $A$ and $\sk$, we derive
\begin{eqnarray}
&&\left(\nabla^2_{v_1,v_2}\Phi_c\right)\hspace{-4pt}(u,w)\nonumber\\
&&\quad=\sk\bigl((\nabla_{u}A)(v_1,v_2)+\rk\bigl(\dF(v_1),\dF(u)\bigr)\dF(v_2),\dF(w)\bigr)\nonumber\\
&&\quad+\sk\bigl((\nabla_{w}A)(v_1,v_2)+\rk\bigl(\dF(v_1),\dF(w)\bigr)\dF(v_2),\dF(u)\bigr)\nonumber\\
&&\quad+\sk\bigl(A(v_1,u),A(v_2,w)\bigr)+\sk\bigl(A(v_1,w),A(v_2,u)\bigr)\nonumber\\
&&\quad-\sind\bigl(\rind(v_1,u)v_2,w\bigr)-\sind\bigl(\rind(v_1,w)v_2,u\bigr).\nonumber
\end{eqnarray}
The decomposition of the tensors $\sk$ and $\rk$, implies
\begin{eqnarray}
&&\sk\Bigl(\rk\bigl(\dF(v_1),\dF(u)\bigr)\dF(v_2),\dF(w)\Bigr)\nonumber\\
&&\quad=(\pi_M^*\gm)\bigl(\rk\bigl(\dF(v_1),\dF(u)\bigr)\dF(v_2),\dF(w)\bigr)\nonumber\\
&&\quad\quad-(\pi_N^*\gn)\bigl(\rk\bigl(\dF(v_1),\dF(u)\bigr)\dF(v_2),\dF(w)\bigr)\nonumber\\
&&\quad=\gm\bigl(\rm\bigl(v_1,u\bigr)v_2,w\bigr)-\gn\bigl(\rn\bigl(\df(v_1),\df(u)\bigr)\df(v_2),\df(w)\bigr)\nonumber\\
&&\quad=\rn\bigl(\df(v_1),\df(u),\df(v_2),\df(w)\bigr)-\rm\bigl(v_1,u,v_2,w\bigr)\,.\nonumber
\end{eqnarray}

Gau\ss' equation (\ref{gauss}) gives
\begin{eqnarray}
&-&\sind\bigl(\rind(v_1,u)v_2,w\bigr)\nonumber\\
&&=-\Phi_c\bigl(\rind(v_1,u)v_2,w\bigr)-\frac{1-c}{1+c}\gind\bigl(\rind(v_1,u)v_2,w\bigr)\nonumber\\
&&=-\Phi_c\bigl(\rind(v_1,u)v_2,w\bigr)+\frac{1-c}{1+c}\rind\bigl(v_1,u,v_2,w\bigr)\nonumber\\
&&=-\Phi_c\bigl(\rind(v_1,u)v_2,w\bigr)\nonumber\\
&&\,\,+\frac{1-c}{1+c}\left\{\gk(A(v_1,v_2),A(u,w))-\gk(A(v_1,w),A(v_2,u))\right\}\nonumber\\
&&\,\,+\frac{1-c}{1+c}\rm(v_1,u,v_2,w)+\frac{1-c}{1+c}\rn(\df(v_1),\df(u),\df(v_2),\df(w)).\nonumber
\end{eqnarray}

In the sequel consider any local orthonormal frame field $\{e_1,\dots,e_m\}$ with respect to the induced
metric $\gind$ on $M$. Then, taking a trace,  we derive the Laplacian of the tensor $\Phi_{c}$.

Let us now summarize the previous computations in the next lemma:

\begin{lemma}\label{laplacephi}
For any smooth map $f:M\to N$, the symmetric tensor $\Phi_c$ satisfies the identity
\begin{eqnarray*}
\bigl(\Delta\Phi_c\bigr)(v,w)&=&\sk\bigl(\nabla_v\vec{H},\dF(w)\bigr)+\sk(\nabla_w\vec{H},\dF(v)\bigr)\\
&&+2\frac{1-c}{1+c}\gk\bigl(\vec{H},A(v,w)\bigr) \\
&&+\Phi_c\bigl(\operatorname{Ric}v,w\bigr)+\Phi_c\bigl(\operatorname{Ric}w,v\bigr) \\
&&+2\sum_{k=1}^m\bigl(\sk-\frac{1-c}{1+c}\gk\bigr)\bigl(A(e_k,v),A(e_k,w)\bigr) \\
&&+\frac{4}{1+c}\sum_{k=1}^m\Bigl(f^*\rn(e_k,v,e_k,w)-c\rm(e_k,v,e_k,w)\Bigr),
\end{eqnarray*}
where
$$\operatorname{Ric}v:=-\sum_{k=1}^m\rind(e_k,v)e_k$$
is the Ricci operator on $(M,\gind)$ and $\{e_1,\dots,e_m\}$ is any orthonormal
frame with respect to the induced metric $\gind$.
\end{lemma}

The expressions of the covariant derivative and the Laplacian of a symmetric
$2$-tensor $\T^{[2]}$ are given in the following Lemma. The proof follows by a straightforward
computation and for that reason will be omitted.

\begin{lemma}\label{laplacesind}
Any symmetric $2$-tensor $\T$ satisfies the identities,
\begin{enumerate}[(i)]
\item
$\nabla_{v}\T^{[2]}=\left(\nabla_{v}\T\right)^{[2]},$
\medskip
\item
$\nabla^2_{v,v}\T^{[2]}=\left(\nabla^2_{v,v}\T\right)^{[2]},$
\medskip
\item
$\Delta\T^{[2]}=\left(\Delta\T\right)^{[2]},$
\end{enumerate}
for any vector $v$ on the manifold $M$.
\end{lemma}

\subsection{Proofs of the Theorems \ref{thmD}, \ref{thmC} and \ref{thmE}}
We will first show the following lemma.
\begin{lemma}\label{lem cd}
Let $f:M\to N$ be weakly length decreasing. Suppose $\{e_1,\dots,e_m\}$ is orthonormal with
respect to $\gind$ such that it diagonalizes the tensor $\sind$. Then for any $e_l$ we have
\begin{eqnarray*}
&&2\sum_{k=1}^m\Bigl(\rm(e_k,e_l,e_k,e_l)-f^*\rn(e_k,e_l,e_k,e_l)\Bigr) \\
&=&
2\sum_{k\neq l}\frac{\lambda_k^2}{1+\lambda_k^2}\Bigl\{\bigl(\sigma-\sigma_N(\df(e_k)\wedge\df(e_l))\bigr)
f^*\gn(e_l,e_l)\\
&&\hspace{4cm}+\sigma\bigl(\gm(e_l,e_l)-f^*\gn(e_l,e_l)\bigr)\Bigr\}\\
&&+\operatorname{Ric}_M(e_l,e_l)-(m-1)\sigma\gm(e_l,e_l)\\
&&+\sum_{k\neq l}\frac{1-\lambda_k^2}{1+\lambda_k^2}\Bigl(\sigma_M(e_k\wedge e_l)+\sigma\Bigr)
\gm(e_l,e_l),
\end{eqnarray*}
where $\operatorname{Ric}_M$ denotes the Ricci curvature with respect to $\gm$, $\sigma_M(e_k\wedge e_l)$
and $\sigma_N(\df(e_k)\wedge \df(e_l))$ are the sectional
curvatures of the planes $e_k\wedge e_l$ on $(M,\gm)$ and  $\df(e_k)\wedge\df (e_l)$ on $(N,\gn)$ respectively.
\end{lemma}
\begin{proof}
In terms of the singular
values we get
$$s(e_k,e_k)=\gm(e_k,e_k)-f^*\gn(e_k,e_k)=\frac{1-\lambda_k^2}{1+\lambda_k^2}.$$
Since
$$1=\gind(e_k,e_k)=\gm(e_k,e_k)+f^*\gn(e_k,e_k)$$
we derive
$$\gm(e_k,e_k)=\frac{1}{1+\lambda_k^2},\quad f^*\gn(e_k,e_k)=\frac{\lambda_k^2}{1+\lambda_k^2}$$
and
$$2\gm(e_k,e_k)=\frac{1-\lambda_k^2}{1+\lambda_k^2}+1,\quad -2f^*\gn(e_k,e_k)=\frac{1-\lambda_k^2}{1+\lambda_k^2}-1.$$
Note also that for any $k \neq l$ we have
$$\gm(e_k,e_l)=f^*\gn(e_k,e_l)=\gind(e_k,e_l)=0.$$
We compute
\begin{eqnarray*}
&&2\sum_{k=1}^m\Bigl(\rm(e_k,e_l,e_k,e_l)-f^*\rn(e_k,e_l,e_k,e_l)\Bigr) \\
&=&2\sum_{k\neq l}\sigma_M(e_k\wedge e_l)\gm(e_k,e_k)\gm(e_l,e_l)\\
&&-2\sum_{k\neq l}\sigma_N(\df(e_k)\wedge \df(e_l))f^*\gn(e_k,e_k)f^*\gn(e_l,e_l).
\end{eqnarray*}
Hence the formula for $\gm(e_k,e_k)$ implies
\begin{eqnarray*}
&&2\sum_{k=1}^m\Bigl(\rm(e_k,e_l,e_k,e_l)-f^*\rn(e_k,e_l,e_k,e_l)\Bigr) \\
&=&\sum_{k\neq l}\left(1+\frac{1-\lambda_k^2}{1+\lambda_k^2}\right)\sigma_M(e_k\wedge e_l)\gm(e_l,e_l)\\
&&+2\sum_{k\neq l}f^*\gn(e_k,e_k)\Bigl\{\bigl(\sigma-\sigma_N(\df(e_k)\wedge\df(e_l))\bigr)
f^*\gn(e_l,e_l)\\
&&\hspace{4cm}+\sigma\bigl(\gm(e_l,e_l)-f^*\gn(e_l,e_l)\bigr)\Bigr\}\\
&&-2\sigma\sum_{k\neq l}f^*\gn(e_k,e_k)\gm(e_l,e_l)\\
&=&\sum_{k\neq l}\left(1+\frac{1-\lambda_k^2}{1+\lambda_k^2}\right)\sigma_M(e_k\wedge e_l)\gm(e_l,e_l)\\
&&+2\sum_{k\neq l}f^*\gn(e_k,e_k)\Bigl\{\bigl(\sigma-\sigma_N(\df(e_k)\wedge\df(e_l))\bigr)
f^*\gn(e_l,e_l)\\
&&\hspace{4cm}+\sigma\bigl(\gm(e_l,e_l)-f^*\gn(e_l,e_l)\bigr)\Bigr\}\\
&&+\sigma\sum_{k\neq l}\left(\frac{1-\lambda_k^2}{1+\lambda_k^2}-1\right)\gm(e_l,e_l).
\end{eqnarray*}
We may then continue to get
\begin{eqnarray*}
&&2\sum_{k=1}^m\Bigl(\rm(e_k,e_l,e_k,e_l)-f^*\rn(e_k,e_l,e_k,e_l)\Bigr) \\
&=&
2\sum_{k\neq l}f^*\gn(e_k,e_k)\Bigl\{\bigl(\sigma-\sigma_N(\df(e_k)\wedge\df(e_l))\bigr)
f^*\gn(e_l,e_l)\\
&&\hspace{4cm}+\sigma\bigl(\gm(e_l,e_l)-f^*\gn(e_l,e_l)\bigr)\Bigr\}\\
&&+\operatorname{Ric}_M(e_l,e_l)-(m-1)\sigma\,\gm(e_l,e_l)\\
&&+\sum_{k\neq l}\frac{1-\lambda_k^2}{1+\lambda_k^2}\Bigl(\sigma_M(e_k\wedge e_l)+\sigma\Bigr)
\gm(e_l,e_l).
\end{eqnarray*}
This completes the proof.
\end{proof}

{\bf Proof of Theorem \ref{thmD}.}
Suppose that $f:M\to N$ is weakly length decreasing.
Then the tensor $\sind$ satisfies
$$\sind=\gm-f^*\gn\ge 0.$$
In case $\sind >0$ the map $f$ is also
strictly area decreasing. Thus in such a case the statement follows from
Theorem \ref{thmC} which we will prove further below. It remains to show that $\sind$
vanishes identically, if $\sind$ admits a null-eigenvalue somewhere.

{\bf Claim 1.} {\it The tensor $\sind$ has a null-eigenvalue everywhere on $M$, if this is
the case in at least one point $x\in M$.}

Since $\sind=\Phi_1$, from Lemma \ref{laplacephi} we get
$$\Delta \sind+\Psi(\sind)=0,$$
where
\begin{eqnarray*}
\bigl(\Psi(\vartheta)\bigr)(v,w)&=&-\vartheta\bigl(\operatorname{Ric}v,w\bigr)-\vartheta\bigl(\operatorname{Ric}w,v\bigr) \\
&&-2\sum_{k=1}^m\sk\bigl(A(e_k,v),A(e_k,w)\bigr) \\
&&+2\sum_{k=1}^{m}\Bigl(\rm(e_k,v,e_k,w)-f^*\rn(e_k,v,e_k,w)\Bigr).
\end{eqnarray*}
Let $v$ be a null-eigenvector of the symmetric, positive semi-definite tensor $\vartheta$.
Since $f$ is weakly length decreasing, equation (\ref{normal}) shows that $\sk$ is non-positive
definite on the normal bundle of the graph. Hence
\begin{eqnarray*}
\bigl(\Psi(\vartheta)\bigr)(v,v)\ge
2\sum_{k=1}^{m}\Bigl(\rm(e_k,v,e_k,v)-f^*\rn(e_k,v,e_k,v)\Bigr)\ge 0,
\end{eqnarray*}
where we have used Lemma \ref{lem cd} and the curvature assumptions on $(M,\gm)$, $(N,\gn)$
respectively. This shows that $\Psi$ satisfies
the null-eigenvector condition  and Claim 1 follows from the strong maximum principle
in Theorem \ref{mp2}.

{\bf Claim 2.} {\it If $\sind$ admits a null-eigenvalue at some point $x\in M$, then $\sind$ vanishes
at $x$.}

We already know that the tensor $\sind$ admits a null-eigenvalue everywhere on $M$.
Since $\sind\ge 0$ we may then apply the test criterion Theorem \ref{test} to the tensor $\sind$
at an arbitrary point $x\in M$.
At $x$ consider a basis $\{e_1,\dots,e_m\}$, orthonormal with respect to $\gind$ consisting of eigenvectors of $\sind$,
such that $v:=e_m$ is a null-eigenvector of $\sind$ and  $\lambda_m^2=1$.
From Lemma \ref{lem cd}, we conclude
\begin{eqnarray}
0&\ge&\bigl(\Psi(\sind)\bigr)(e_m,e_m)\nonumber\\
&\ge&2\sum_{k\neq m}\frac{\lambda_k^2}{1+\lambda_k^2}\Bigl\{\bigl(\underbrace{\sigma-\sigma_N(\df(e_k)\wedge\df(e_m))}_{\ge 0}\bigr)
\underbrace{f^*\gn(e_l,e_l)}_{\ge 0}\nonumber\\
&&\hspace{4cm}+\sigma\bigl(\underbrace{\gm(e_m,e_m)-f^*\gn(e_m,e_m)}_{= 0}\bigr)\Bigr\}\nonumber\\
&&+\underbrace{\operatorname{Ric}_M(e_m,e_m)-(m-1)\sigma\,\gm(e_m,e_m)}_{\ge 0}\nonumber\\
&&+\sum_{k\neq m}\underbrace{\frac{1-\lambda_k^2}{1+\lambda_k^2}}_{\ge 0}\Bigl(\underbrace{\sigma_M(e_k\wedge e_m)+\sigma}_{>0}\Bigr)\underbrace{\gm(e_m,e_m)}_{=\frac{1}{2}}=0,\label{eq vanish}
\end{eqnarray}
because the curvature assumptions on $(M,\gm)$ and $(N,\gn)$ imply that the
right hand side is a sum of non-negative terms and thus we conclude that all of them must vanish.
In particular,
$$\sigma_M(e_k\wedge e_m)+\sigma>0$$
implies $\lambda_k^2=1$ for all $k$.

Now we can finish the proof of Theorem \ref{thmD}. Claim 1 and 2 imply that a weakly
length decreasing map $f$ which is not strictly length decreasing must be an isometric immersion.
Once we know that all tangent vectors at $x$ are
null-eigenvectors of $\sind$, we may choose $e_m$ in (\ref{eq vanish}) arbitrarily. Then
$$\operatorname{Ric}_M(v,v)=(m-1)\sigma\gm(v,v)$$
and
$$\sigma=\sigma_N\bigl(\df(v),\df (w)\bigr)$$
for all linearly independent vectors $v,w\in T_xM$. This completes the proof of Theorem \ref{thmD}.
\hfill$\square$

{\bf Proof of Theorem \ref{thmC}.}
Since the manifold $M$ is compact, there exists a point $x_{0}$ where the smallest
eigenvalue of $\sind^{[2]}$ with respect to the metric $\Gind$ attains its minimum.
Let us denote this value by $\rho_{0}$. Note that in terms of the singular values
$$\lambda_1^2\le\cdots\le\lambda_m^2$$
we must have
$$\rho_{0}=\frac{1-\lambda^{2}_{m}(x_0)}{1+\lambda^{2}_{m}(x_0)}
+\frac{1-\lambda^{2}_{m-1}(x_0)}{1+\lambda^{2}_{m-1}(x_0)}\ge 0.$$
For simplicity we set
$$\kappa:=\lambda^{2}_{m-1}(x_0)\quad\text{and}\quad\mu:=\lambda^2_{m}(x_0).$$
Hence,
$$\rho_0=2\frac{1-\kappa\mu}{(1+\kappa)(1+\mu)}\,.$$
{\bf Claim 3.}\textit{ If $\mu=0$, then the map $f$ is constant.}

In this case we have $\rho_0=2$. Because $\rho_0$ is the minimum of the smallest eigenvalue
of the symmetric tensor $\sind^{[2]}$, we obtain
\begin{equation*}
1\le \frac{1-\lambda^2_{i}(x)\lambda^2_{j}(x)}{(1+\lambda^2_i(x))(1+\lambda^2_j(x))},
\end{equation*}
for any $x\in M$ and $1\le i<j\le m$. From the above inequality one can readily see that
all the singular values of $f$ vanish everywhere. Thus, in this case $f$ is constant.
This completes the proof of Claim $3$.

Since we are assuming that $f$ is weakly area decreasing, we deduce that $\kappa\mu\le1$.
Consider now the symmetric $2$-tensor
$$\Phi:=\Phi_{\frac{2-\rho_{0}}{2+\rho_0}}=\sind-\frac{\rho_0}{2}\gind.$$
According to Lemma \ref{laplacesind},
$$\Delta\Phi^{[2]}=\left(\Delta\Phi\right)^{[2]}.$$
At $x_0$ consider an orthonormal bases $\{e_1,\dots,e_m\}$
with respect to $\gind$ such that
$\sind$ becomes diagonal and
$$\sind(e_k,e_k)=\frac{1-\lambda_k^2}{1+\lambda_k^2}.$$
According to Theorem \ref{test},
we obtain
\begin{eqnarray*}
0&\le&\left(\Delta\Phi\right)^{[2]}\left( e_{m-1}\wedge  e_m,
 e_{m-1}\wedge  e_m\right) \\
&=&\left(\Delta\Phi\right)\left( e_{m-1}, e_{m-1}\right)
+\left(\Delta\Phi\right)\left( e_m, e_m\right).
\end{eqnarray*}
In view of Lemma \ref{laplacephi}, we deduce that
\begin{eqnarray}
0&\le&2\Phi(\Ric  e_{m-1}, e_{m-1})+2\Phi(\Ric  e_m, e_m) \nonumber\\
&&+2\sum^{m}_{k=1}\Bigl(\sk-\frac{\rho_0}{2}\gk\Bigr)(A( e_k, e_{m-1}),A( e_k, e_{m-1}))\nonumber \\
&&+2\sum^{m}_{k=1}\Bigl(\sk-\frac{\rho_0}{2}\gk\Bigr)(A( e_k, e_m),A( e_k, e_m))\nonumber \\
&&+(2+\rho_0)\sum^{m}_{k=1}f^*\rn( e_k,e_{m-1}, e_k, e_{m-1}) \nonumber \\
&&-(2-\rho_0)\sum^{m}_{k=1}\rm( e_k, e_{m-1}, e_k, e_{m-1}) \nonumber \\
&&+(2+\rho_0)\sum^{m}_{k=1}f^*\rn(e_k, e_m, e_k, e_m) \nonumber \\
&&-(2-\rho_0)\sum^{m}_{k=1}\rm( e_k, e_m, e_k, e_m).\label{inequality5}
\end{eqnarray}
Because $ e_{m}$ is an eigenvector of $\sind$ with respect to $\gind$, we have
\begin{eqnarray*}
\Phi(\Ric  e_{m}, e_{m})=\frac{\kappa-\mu}{(1+\kappa)(1+\mu)}\gind(\Ric  e_{m}, e_{m}).
\end{eqnarray*}
From the Gau{\ss} equation (\ref{gauss}) and the minimality of the graph, we obtain that
\begin{eqnarray*}
\gind(\Ric  e_{m}, e_{m})&=&\sum^{m}_{k=1}\rm( e_k, e_{m}, e_k, e_{m}) \\
&&+\sum^{m}_{k=1}f^*\rn(e_k,e_m,e_k,e_m) \\
&&-\sum^{m}_{k=1}\gk(A( e_k, e_{m}),A( e_k, e_{m})).
\end{eqnarray*}
Hence,
\begin{eqnarray}
\Phi(\Ric  e_m, e_m)&=&\tfrac{\kappa-\mu}{(1+\kappa)(1+\mu)}\tsum\limits^{m}_{k=1}
\rm( e_k, e_m, e_k, e_m) \label{phi1}\\
&+&\tfrac{\kappa-\mu}{(1+\kappa)(1+\mu)}\tsum\limits^{m}_{k=1}
f^*\rn( e_k, e_m, e_k, e_m) \nonumber \\
&-&\tfrac{\kappa-\mu}{(1+\kappa)(1+\mu)}\tsum\limits^{m}_{k=1}\gk(A( e_k, e_m),A( e_k, e_m))\nonumber.
\end{eqnarray}
Similarly,
\begin{eqnarray}
\Phi(\Ric  e_{m-1}, e_{m-1})&=&\tfrac{\mu-\kappa}{(1+\kappa)(1+\mu)}\tsum\limits^{m}_{k=1}\rm
( e_k, e_{m-1}, e_k, e_{m-1})\label{phi2} \\
&+&\tfrac{\mu-\kappa}{(1+\kappa)(1+\mu)}\tsum\limits^{m}_{k=1}f^*\rn
( e_k, e_{m-1}, e_k, e_{m-1})\nonumber \\
&-&\tfrac{\mu-\kappa}{(1+\kappa)(1+\mu)}\tsum\limits^{m}_{k=1}\gk(A(e_k, e_{m-1}),A( e_k, e_{m-1})).\nonumber
\end{eqnarray}
In view of (\ref{phi1}) and (\ref{phi2}), the inequality (\ref{inequality5}) can be now written equivalently in
the form
\begin{eqnarray}
0&\le&\sum^{m}_{k=1}(\sk-\frac{1-\mu}{1+\mu}\gk)(A( e_k, e_m),A( e_k, e_m))\nonumber \\
&&+\sum^{m}_{k=1}(\sk-\frac{1-\kappa}{1+\kappa}\gk)(A( e_k, e_{m-1}),A( e_k, e_{m-1}))\nonumber \\
&&+\frac{2}{1+\mu}\sum^{m}_{k=1}\left(f^*\rn-\mu\rm\right)( e_k, e_m, e_k, e_m)\nonumber \\
&&+\frac{2}{1+\kappa}\sum^{m}_{k=1}\left(f^*\rn-\kappa\rm\right)( e_k, e_{m-1}, e_k, e_{m-1}).\label{inequality6}
\end{eqnarray}

\textbf{Claim 4.}\textit{ The sum $\mathcal{A}$ of the first two terms on the right hand side of inequality $(\ref{inequality6})$
is non-positive.}

Indeed, if $\mu=0$, then $f$ is constant by Claim $3$ and thus $\mathcal{A}=0$. So, let us consider the case where $\mu>0$.
From Theorem \ref{test} again, we have
\begin{eqnarray}
0&=&(\nabla_{e_k}(\sind^{[2]}-\rho_0\Gind))( e_m\wedge e_{m-1}, e_m\wedge e_{m-1}) \nonumber\\
&=&2(\nabla_{ e_k}\sind)( e_m, e_m)+2(\nabla_{ e_k}\sind)( e_{m-1}, e_{m-1})\nonumber\\
&=&4\sk(A( e_k, e_m), e_m)+4\sk(A( e_k, e_{m-1}), e_{m-1})\label{eq zero}
\end{eqnarray}
for any $k$. Since $\dim(N)=1$ implies that $\operatorname{rank}(\df)\le 1$, from (\ref{mixed}) we obtain
\begin{eqnarray*}
0&=&A_{\xi_n}(e_k,e_m)\sk(\xi_n,e_m)+A_{\xi_n}(e_k,e_{m-1})\underbrace{\sk(\xi_n,e_{m-1})}_{=0}\\
&=&-2A_{\xi_n}(e_k,e_m)\underbrace{\frac{\sqrt{\mu}}{1+\mu}}_{>0},
\end{eqnarray*}
where here
$$A_{\xi}(v,w):=\gn(A(v,w),\xi),\quad v,w\in T_xM,$$
stands for the second fundamental form of the graph $\Gamma(f)$ in the normal direction $\xi$
and the normal basis $\{\xi_1,\dots,\xi_n\}$ is chosen as in (\ref{normal}).
Hence, by the weakly area decreasing property of $f$, we get
$$\mathcal{A}=-\sum_{k=1}^m\left(\frac{1-\mu}{1+\mu}+\frac{1-\kappa}{1+\kappa}\right)A_{\xi_n}^2(e_k,e_{m-1})\le 0.$$

In case $\dim(N)\ge 2$, from (\ref{normal}), (\ref{eq zero}) and the weakly area decreasing condition we obtain
\begin{eqnarray*}
\mathcal{A}&=&\sum^{m}_{k=1}(\sk-\frac{1-\mu}{1+\mu}\gk)(A( e_k, e_m),A( e_k, e_m)) \\
&&+\sum^{m}_{k=1}(\sk-\frac{1-\kappa}{1+\kappa}\gk)(A( e_k, e_{m-1}),A( e_k, e_{m-1})) \\
&\le&-2\frac{1-\mu}{1+\mu}\sum^{m}_{k=1}A^2_{\xi_n}( e_k, e_m)
-2\frac{1-\kappa}{1+\kappa}\sum^{m}_{k=1}A^2_{\xi_{n-1}}( e_k, e_{m-1}).
\end{eqnarray*}
In view of equations (\ref{eq zero}) and (\ref{mixed}), we have
\begin{eqnarray*}
0&=&\sk(A( e_k, e_m), e_m)+\sk(A( e_k, e_{m-1}), e_{m-1}) \\
&=&-2\frac{\sqrt{\mu}}{1+\mu}A_{\xi_n}( e_k, e_m)-2\frac{\sqrt{\kappa}}{1+\kappa}A_{\xi_{n-1}}( e_k, e_{m-1}).
\end{eqnarray*}
Hence,
$$A^2_{\xi_n}( e_k, e_m)=\frac{\kappa(1+\mu)^2}{\mu(1+\kappa)^2}A^2_{\xi_{n-1}}( e_k, e_{m-1}).$$
Because, $\kappa\le\mu$ and $\kappa\mu\le1$, we deduce that
$$\frac{\kappa(1+\mu)^2}{\mu(1+\kappa)^2}\le 1.$$
This proves our assertion. Now it is clear that the quantity $\mathcal{A}$ is always non-positive which
proves Claim $4$.

{\bf Claim 5.} \textit{The sum $\mathcal{B}$ of the last two terms on the right hand side
of inequality (\ref{inequality6}) is non-positive}.

We have,
\begin{eqnarray*}
\mathcal{B}&=&\frac{1}{1+\mu}\underbrace{\sum^{m}_{k=1}2\left(f^*\rn-\mu\rm\right)
( e_k, e_m, e_k, e_m)}_{=:\mathcal{B}_1} \nonumber \\
&&+\frac{1}{1+\kappa}\underbrace{\sum^{m}_{k=1}2\left(f^*\rn-\kappa\rm\right)
( e_k, e_{m-1}, e_k, e_{m-1})}_{=:\mathcal{B}_2}\nonumber
\end{eqnarray*}
From the identities (\ref{met2}) and (\ref{met4}), we deduce that
$$\gm=\frac{1}{2}(\gind+\sind)\quad\text{and}\quad f^*\gn=\frac{1}{2}(\gind-\sind).$$
Since $\{ e_1,\dots, e_{m}\}$ diagonalizes $\gind$ and $\sind$, it follows that it
diagonalizes $\gm$ and $f^*\gn$ as well. In fact, for any $i\in\{1,\dots,m\}$, we have
$$f^*\gn( e_{i}, e_{i})=\lambda^2_{i}\gm( e_{i}, e_{i}).$$
Proceeding exactly as in the proof of Lemma \ref{lem cd}, but using $\mu\sigma_{M}$ instead of $\sigma_M$ and
$\mu\sigma$ instead of $\sigma$, we obtain that
\begin{eqnarray*}
\mathcal{B}_1&=&2\sum_{k\neq m}\Bigl(f^*\rn(e_k,e_m,e_k,e_m)-\mu\rm(e_k,e_m,e_k,e_m)\Bigr) \\
&=&-\frac{2\mu}{1+\mu}\sum_{k\neq m}f^*\gn(e_k,e_k)\Bigl(\sigma-\sigma_N(\df(e_k)\wedge\df(e_m))\Bigr)\\
&&-\mu\Bigl(\operatorname{Ric}_M(e_m,e_m)-(m-1)\sigma\,\gm(e_m,e_m)\Bigr)\\
&&-\frac{\mu}{1+\mu}\sum_{k\neq m}\frac{1-\lambda_k^2}{1+\lambda_k^2}\Bigl(\sigma_M(e_k\wedge e_{m})+\sigma\Bigr).
\end{eqnarray*}
Similarly,
\begin{eqnarray*}
\mathcal{B}_2&=&2\sum\limits_{k\neq m-1}\Bigl(f^*\rn(e_k,e_{m-1},e_k,e_{m-1})-\kappa\rm(e_k,e_{m-1},e_k,e_{m-1})\Bigr) \\
&=&-\frac{2\kappa}{1+\kappa}\sum_{k\neq m-1}f^*\gn(e_k,e_k)\Bigl(\sigma-\sigma_N(\df(e_k)\wedge\df(e_{m-1}))\Bigr)\\
&&-\kappa\Bigl(\operatorname{Ric}_M(e_{m-1},e_{m-1})-(m-1)\sigma\,\gm(e_{m-1},e_{m-1})\Bigr)\\
&&-\frac{\kappa}{1+\kappa}\sum_{k\neq m-1}\frac{1-\lambda_k^2}{1+\lambda_k^2}\Bigl(\sigma_M(e_k\wedge e_{m-1})+\sigma\Bigr).
\end{eqnarray*}
Taking into account that $\lambda^{2}_{1}\le\cdots\le\lambda^{2}_{m-2}\le 1$, we deduce that
\begin{eqnarray}
\mathcal{B}&=&-\frac{2\mu}{(1+\mu)^2}\sum_{k\neq m}f^*\gn(e_k,e_k)\Bigl(\sigma-\sigma_N(\df(e_k)\wedge\df(e_m))\Bigr) \nonumber\\
&&-\frac{2\kappa}{(1+\kappa)^2}\sum_{k\neq m-1}f^*\gn(e_k,e_k)\Bigl(\sigma-\sigma_N(\df(e_k)\wedge\df(e_{m-1}))\Bigr)\nonumber\\
&&-\frac{\mu}{1+\mu}\Bigl(\operatorname{Ric}_M(e_{m},e_{m})-(m-1)\sigma\,\gm(e_{m},e_{m})\Bigr)\nonumber\\
&&-\frac{\kappa}{1+\kappa}\Bigl(\operatorname{Ric}_M(e_{m-1},e_{m-1})-(m-1)\sigma\,\gm(e_{m-1},e_{m-1})\Bigr)\nonumber\\
&&-\frac{\mu}{(1+\mu)^2}\sum_{k=1}^{m-2}\frac{1-\lambda_k^2}{1+\lambda_k^2}\Bigl(\sigma_M(e_k\wedge e_{m})+\sigma\Bigr)\nonumber\\
&&-\frac{\kappa}{(1+\kappa)^2}\sum_{k=1}^{m-2}\frac{1-\lambda_k^2}{1+\lambda_k^2}\Bigl(\sigma_M(e_k\wedge e_{m})+\sigma\Bigr)\nonumber \\
&&-\frac{(\kappa+\mu)(1-\kappa\mu)}{(1+\kappa)(1+\mu)}\Bigl(\sigma_M(e_{m-1}\wedge e_{m})+\sigma\Bigr)\nonumber\\
&&\le 0\label{inequality7}.
\end{eqnarray}
This completes the proof of Claim $5$.

Now we shall distinguish two cases.

{\bf Case 1.} Assume at first that $f$ is strictly area decreasing. Hence, $\kappa\mu<1$. In view of our
curvature assumptions, Claim $4$, Claim $5$, (\ref{inequality7}) and from inequality (\ref{inequality6}) we deduce
that $\kappa=\mu=0$. Hence, from Claim $3$ the map $f$ must be constant.

{\bf Case 2.} Suppose now that there exists a point $x_0\in M$ such that
$\|\Lambda^2\df\|(x_0)=1$.
In this case we have that $\kappa\mu=1$. From Claim $1$, (\ref{inequality7}) and
inequality (\ref{inequality6}) we deduce that at $x_0$ we must have
$$1=\lambda^{2}_{1}(x_0)=\cdots=\lambda^{2}_{m-2}(x_0)\le\kappa\le 1.$$
Hence, $\kappa=1$ and so $\mu=1$. Therefore, at each point $x$ where $\Phi^{[2]}$
has a zero eigenvalue, all the singular values of $f$ are equal to $1$. Thus, the set
$$D:=\{x\in M:\|\Lambda^2\df\|=1\},$$
is closed, non-empty and moreover $D=\{x\in M:f^{\ast}\gn=\gm\}.$
Obviously, the map $f$ is strictly area decreasing on the complement of $D$. Moreover,
by (\ref{inequality7}), $\Ric_{M}=(m-1)\sigma$ at any point of $D$ and the restriction of $\sigma_{N}$
to $\df(TD)$ is equal to $\sigma$.

This completes the proof of Theorem \ref{thmC}.
\hfill$\square$

{\bf Proof of Theorem \ref{thmE}.}
Note that in this case the singular values of the map $f$ are
$$0=\lambda^{2}_{1}=\cdots=\lambda^{2}_{m-1}=\kappa\le\mu.$$
Hence, automatically, $f$ is strictly area decreasing. From Claim $4$, Claim $5$,
inequality (\ref{inequality5}) and (\ref{inequality7}), we deduce that
$$0\le-2\mu{\Ric}_{M}(e_m,e_m)\le 0.$$
Thus $\mu=0$ and $f$ is a constant map. This completes the proof of Theorem \ref{thmE}.
\hfill$\square$

\subsection{Final remarks}
We end this paper with examples and remarks concerning the imposed assumptions
in Theorems \ref{thmD}, \ref{thmC} and \ref{thmE}.

\begin{remark}
In several cases, graphical submanifolds over $(M,\gm)$ with \textit{parallel mean curvature}, i.e.,
$$\nabla^{\perp}H=0,$$
where $\nabla^{\perp}$ stands for the connection of the normal bundle, must be minimal.
This problem was first considered by Chern in \cite{c}.  So, whenever graphs with parallel mean curvature
vector are minimal we can immediately apply Theorems \ref{thmD}, \ref{thmC} and \ref{thmE}.
For example this can be done for graphs considered in the paper by G. Li and I.M.C. Salavessa \cite{li}.
\end{remark}

\begin{remark}
The reason that the result of Theorem \ref{thmC} is weaker than that of Theorem \ref{thmD} is
due to the fact that in Theorem \ref{thmC} we cannot apply the strong elliptic maximum principle
stated in Theorem \ref{mp2}. In fact, the null-eigenvector condition of the corresponding
tensor $\Psi(\vartheta^{[2]})$ in the equation of $\Delta \sind^{[2]}$ seems to hold only for some
weakly $2$-positive definite tensors $\vartheta$, including $\sind$.
\end{remark}

\begin{remark}
In some situations, a minimal map $f:M\to N$ satisfying the assumptions in Theorem \ref{thmC}
can only be constant.  For instance, if $\dim M>\dim N$ the map $f$ cannot be an isometric immersion
since $\operatorname{rank}(\df)<\dim M$. Moreover, if $M$ is not Einstein or the sectional curvature
of $N$ is strictly less than $\sigma$, then any such map must be constant.
\end{remark}

\begin{remark}
In this remark we show that the assumptions on the curvatures of $M$ and $N$ in Theorems
\ref{thmD} and \ref{thmC} are sharp.
\begin{enumerate}[i)]
\item{\bf Scaling.}
Suppose that $f:M\to N$ is a smooth map between two Riemannian manifolds
$(M,\gm)$ and $(N,\gn)$, and assume that there exists a constant $c>0$ such that
$f^*\gn<c\,\gm.$
Clearly such a constant exists, if $M$ is compact. Define the rescaled metrics
$$\widetilde{\gind}_M:=c\gm\,,\quad\widetilde{\gind}_N:=c^{-1}\gn\,.$$
One can verify that $f$ is a length (and obviously area) decreasing map with respect to the Riemannian metrics
$\widetilde{\gind}_M$ and $\gn$, as well as with respect to the metrics $\gm$ and $\widetilde{\gind}_N$.
Thus, any smooth map can be made a length decreasing map, if either the domain or the target is
scaled appropriately.
\item
{\bf Totally geodesic maps.}
There are plenty of  non constant length decreasing minimal maps.
For instance, assume that $(M,\gm)$ is a Riemannian manifold and $c\in(0,1)$ a real constant. The
identity map $\Id:(M,\gm)\to(M,c^{-1}\gm)$ gives a length decreasing minimal map whose graph
$\Gamma(\Id)$ is even totally geodesic. If $\sigma_{M}$ and $\sigma_{N}$ are the sectional
curvatures of $(M,\gm)$ and $(N,c^{-1}\gm)$, respectively, then
$$\sigma_{N}=c^{-1}\sigma_{M}>\sigma_{M}.$$
Consequently, Theorems \ref{thmD} and \ref{thmC} are not valid if we assume $\sigma_{N}>\sigma_{M}$. Moreover, the assumption
$\sigma>0$ is essential in these theorems and cannot be removed. Indeed, consider the flat $2$-dimensional
torus $(\mathbb{T}^2,\gind_{\mathbb{T}})$. By scaling properly the metric
$\gind_{\mathbb{T}}$, the
identity map $\Id:\mathbb{T}^2\to\mathbb{T}^2$ produces a length decreasing map. On the other hand, the
 scaled metric is again flat and $\Id$ is certainly neither constant nor an isometry.
\end{enumerate}
\end{remark}

\begin{example}
This example shows that there exists an abundance of length decreasing minimal maps that
are not totally geodesic.
\begin{enumerate}[i)]
\item{\bf Holomorphic maps.}
According to the
Schwarz-Pick {Lem\-ma}, any non-linear holomorphic map of the unit disc $D$ in the complex plane
$\mathbb{C}$ to itself is strictly length decreasing with respect to the Poincar\'e metric.
The holomorphicity implies that $f$ is a minimal map (cf., \cite{eells}). On the
other hand, L. Ahlfors \cite{ahlfors} exposed in his generalization of the Schwarz-Pick Lemma the
essential role played by the curvature. He proved that if $f:M\to N$ is a holomorphic map, where
$N$ is a Riemann surface with a metric $\gn$ whose Gaussian curvature is bounded from above by a
negative constant $-b$ and $M:=D$ is the unit disc in $\mathbb{C}$ endowed with
an invariant metric $\gm$ whose Gaussian curvature is a negative constant $-a$, then
$$f^{*}\gn\leq \frac{a}{b}\gm.$$
Ahlfors' result was extended by S.T. Yau \cite{yau1} for holomorphic maps between
complete K\"{a}hler manifolds. More precisely, Yau showed that any holomorphic map $f:M\to N$,
where here $M$ is a complete K\"{a}hler manifold with Ricci curvature bounded from below by a negative
constant $-a$ and $N$ is a Hermitian manifold with holomorphic bisectional curvature  bounded
from above by a negative constant $-b$, then
$f^{*}\gn\leq\frac{a}{b}\gm.$
\item{\bf Biholomorphic maps.}
Let $M$ be a K\"{a}hler manifold and $\Aut(M)$ its \textit{automorphism group}, that is
the group of all biholomorphic maps of $M$. When $m\ge 4$, the group $\Aut(M)$ can be
arbitrary large (cf. \cite{akhiezer}). This indicates that the results of Theorem \ref{thmC},
cannot be extended for the $m$-Jacobian $\Lambda^{m}\df$. For example, let $M$ be
compact, $y_0$ a fixed point on $M$, and $f\in\Aut(M)$. Then, the map $\tilde{f}:M\times M\to M\times M$,
$\tilde{f}(x,y)=(f(x),y_0),$
is minimal, as holomorphic, and has identically zero $m$-Jacobian.
In the flat case we can give even explicit
examples. For instance, consider the map $f:\mathbb{C}^2=\real{4}\to\mathbb{C}^2=\real{4}$, given by
$$f(z,w):=(\beta z+h(w),w),\quad z,w\in\mathbb{C},$$
where $h:\mathbb{C}\to\mathbb{C}$ is a non-affine holomorphic map and $\beta\le 1$ a positive
real number. Note that the graph $\Gamma(f)$ is minimal in $\real{8}$, $\|\Lambda^{4}\df\|=\beta\le 1$
and $f$ is certainly not an isometry.
\end{enumerate}
\end{example}

\begin{remark}
Let $M$ and $N$ be two Riemannian manifolds satisfying the curvature assumptions in Theorem \ref{thmC}.
Following essentially the same computations as in the proof of Theorem \ref{thmD}, we can prove that the
strictly area decreasing property of a map $f:M\to N$ is preserved under mean
curvature flow. The convergence shall be explored in
another article where we shall also derive a parabolic analogue of Theorem \ref{mp1}.
\end{remark}

{\small{\bf{Acknowledgments:}}{ The first author would like to express his gratitude to the Max-Planck Institut
f\"{u}r Mathematik in den Nauturwissenschaften Leipzig and especially to Professor J. Jost for the scientific
support and everything that he benefited during the stay at the Institute. Moreover, he would like
to thank Dr. B. Hua for many stimulating conversations.}}

\bibliographystyle{y2k}

\end{document}